\documentclass{amsart}
\usepackage[english]{babel}
\usepackage{amssymb}
\usepackage{amsmath}
\usepackage{amsthm}
\usepackage[left=3.5cm, right=3.5cm, top=2cm, bottom=3cm]{geometry}
\usepackage{mathtools}
\usepackage{xcolor}
\usepackage{tikz-cd}
\usepackage{amscd}
\usepackage{pb-diagram}
\usepackage{comment}

\newcommand{\fgl}{\mathfrak{gl}}
\newcommand{\fh}{\mathfrak{h}}
\newcommand{\fz}{\mathfrak{z}}
\newcommand{\bc}{\mathbb{C}}
\newcommand{\bp}{\mathbb{P}}
\newcommand{\br}{\mathbb{R}}
\newcommand{\bz}{\mathbb{Z}}
\newcommand{\mA}{\mathcal{A}}

\newcommand{\gll}{\mathfrak{gl}_2}
\newcommand{\gln}{\mathfrak{gl}_n}
\newcommand{\yt}{Y(\mathfrak{gl}_2)}
\newcommand{\yn}{Y(\mathfrak{sl}_n)}
\newcommand{\yg}{Y(\mathfrak{g})}

\newcommand{\sln}{\mathfrak{sl}_n}
\newcommand{\fsl}{\mathfrak{sl}}
\newcommand{\bC}{\mathbb C^n}
\newcommand{\tr}{\operatorname{tr}}

\newcommand{\Tp}{T^{(p)}}
\newcommand{\cTp}{\mathcal T^{(p)}}
\newcommand{\cTone}{\mathcal T^{(1)}}

\newcommand{\bw}[1]{\bigwedge^{\raisebox{-0.4ex}{\scriptsize\(#1\)}}}

\newcommand{\id}{\operatorname{id}}
\newcommand{\Flip}{\mathrm{Flip}}
\newcommand{\End}{\operatorname{End}}
\newcommand{\cR}{\mathcal{R}}

\DeclareMathOperator{\ev}{ev}
\DeclareMathOperator{\gr}{gr}

\DeclareMathOperator{\Hom}{Hom}

\DeclareMathOperator{\qdet}{qdet}
\DeclareMathOperator{\sym}{sym}
\DeclareMathOperator{\diag}{diag}

\newtheorem{proposition}{Proposition}[section]
\newtheorem{lemma}[proposition]{Lemma}
\newtheorem{theorem}[proposition]{Theorem}
\newtheorem{corol}[proposition]{Corollary}

\newtheorem*{remark}{Remark}
\newtheorem{conj}[proposition]{Conjecture}

\theoremstyle{definition}
\newtheorem{definition}{Definition}[section]

\title{Spectra of Bethe subalgebras of $Y(\mathfrak{gl}_n)$ in tame representations}
\author{Aleksei Ilin, Inna Mashanova-Golikova and Leonid Rybnikov}

\begin{document}
\maketitle

\begin{abstract} We study the eigenproblem for Bethe subalgebras of the Yangian $Y(\mathfrak{gl}_n)$ in tame representations, i.e. in finite dimensional representations which admit Gelfand-Tsetlin bases. Namely, we prove that for any tensor product of skew modules $V=\otimes_{i=1}^k V_{\lambda_i \setminus \mu_i}(z_i)$ over the Yangian $Y(\mathfrak{gl}_n)$ with generic $z_i$'s, the family of Bethe subalgebras $B(X)$ with $X$ being a regular element of the maximal torus of $GL_n$ (or, more generally, with $X \in \overline{M_{0,n+2}}$) acts with a cyclic vector on $V$. Moreover, for $X$ in the real form of $\overline{M_{0,n+2}}$ which is the closure of regular unitary diagonal matrices we show, that the family of subalgebras $B(X)$ acts with simple spectrum on $\otimes_{i=1}^k V_{\lambda_i \setminus \mu_i}(z_i)$ for generic $z_i$'s where all $V_{\lambda_i \setminus \mu_i}(z_i)$ are Kirillov-Reshetikhin modules. In the subsequent paper we will use this to define a KR-crystal structure on the spectrum of a Bethe subalgebra on $V$.
\end{abstract}


\section{Introduction}

\subsection{Yangian and Bethe subalgebras}
The Yangian $Y(\fgl_n)$ is a Hopf algebra, historically one of the first examples of {\em quantum groups}. $Y(\fgl_n)$ is in certain sense unique Hopf algebra deforming the enveloping algebra $U(\fgl_n[t])$ (see \cite{drin}), where
$\fgl_n[t]$ is the Lie algebra of $\fgl_n$-valued polynomials. There is an action of the additive group $\bc$ by automorphisms on $Y(\fgl_n)$ deforming the action of $\bc$ on $U(\fgl_n[t])$ which shifts the variable $t$. For the details and links on $Y(\fgl_n)$, we refer the reader to the book \cite{M} by A.~Molev. 

There is a flat family of maximal commutative subalgebras $B(C)\subset Y(\fgl_n)$, called \emph{Bethe subalgebras}, parameterized by invertible diagonal matrices $C\in GL_n$ with pairwise different eigenvalues see e.g. \cite{nazol} which are stable under the $\bc$-action by shift automorphisms of $Y(\fgl_n)$. This family originates from the integrable models in statistical mechanics and algebraic Bethe ansatz. More precisely, the image of $B(C)$ in a tensor product of evaluation representations of $Y(\fgl_n)$ form a complete set of Hamiltonians of the XXX Heisenberg magnet chain, cf. \cite{baxter,kbi}.

According to \cite{ir2} the family of Bethe subalgebras extends to a bigger family $B(X)$ of commutative subalgebras in $Y(\fgl_n)$ with $X$ taking values in the Deligne-Mumford space $\overline{M_{0,n+2}}$ of stable rational curves with $n+2$ marked points: the subalgebra $B(C)$ corresponds to the non-degenerate rational curve with the marked points being $0,\infty$ and the eigenvalues of $C$, but there are also some new subalgebras corresponding to degenerate curves $X\in\overline{M_{0,n+2}}$. 
In particular, the subalgebra corresponding to the most degenerate \emph{caterpillar} curve corresponds to the \emph{Cartan} subalgebra $H\subset Y(\fgl_n)$, also known as the \emph{Gelfand-Tsetlin} subalgebra, generated by all centers of the smaller Yangians $Y(\fgl_1) \subset Y(\fgl_2) \subset \ldots \subset Y(\fgl_n)$ embedded in the standard way.

\subsection{Representations of $Y(\fgl_n)$}
It is possible to obtain Yangian representations from representations of $\fgl_n$ using the \emph{evaluation homomorphism} 
$$\ev: Y(\fgl_n) \to U(\fgl_n)$$
which gives a structure of $Y(\fgl_n)$ module on every $\fgl_n$-module called \emph{evaluation} $Y(\fgl_n)$-module. Also, from the $\bc$-action on $Y(\fgl_n)$, for any $z\ \in \bc$ we have an automorphism $\tau_z: Y(\fgl_n) \to Y(\fgl_n)$ (see Section \ref{def} for precise definitions), so we can twist any of the evaluation modules by such automorphism.
Namely, for any irreducible $\fgl_n$-module $V_{\lambda}$, we denote by $V_{\lambda}(z)$ the $Y(\fgl_n)$-module with the underlying space $V_\lambda$ and the $Y(\fgl_n)$ action given by $\ev \circ \tau_z$.

It is possible to generalize this construction of $Y(\fgl_n)$-modules using the \emph{centralizer construction} of the Yangian, due to Olshansky \cite{olsh}. Namely, consider the embedding $\fgl_k\subset \fgl_{n+k}$ as the subalgebra of lower-right block $k\times k$-matrices, then for any $k \geq 0$ there is a homomorphism
$$\eta_k: Y(\fgl_n) \to U(\fgl_{n+k})^{\fgl_k},$$ which is surjective modulo the center of $U(\fgl_{n+k})$
(in particular, we have $\eta_0 = \ev$). Let $V_{\lambda}$ be an irreducible representation of $\fgl_{n+k}$ with the highest weight $\lambda$. Consider the restriction of $V_{\lambda}$ to $\fgl_k$:
$$V_{\lambda} = \bigoplus_{\mu} M_{\lambda \mu} \otimes V_{\mu},$$ where $M_{\lambda \mu} : = \Hom(V_{\mu}, V_{\lambda})$ is the multiplicity space with action of $U(\fgl_{n+k})^{\fgl_k}$ and therefore is an irreducible representation of $Y(\fgl_n)$. Representations of this form are called {\em skew representation} of $Y(\fgl_n)$ because they depend on the \emph{skew Young tableau} $\lambda\setminus\mu$. 
If $M_{\lambda \mu}$ is any skew representation of $Y(\fgl_n)$ then we denote by $V_{\lambda \setminus \mu}(z)$ the (irreducible) representation where the action of $Y(\fgl_n)$ is given by $\eta_k \circ \tau_z$. We also call these representations {\it skew} representations of $Y(\fgl_n)$.

We also note that Yangian is a Hopf algebra. This allows us to consider tensor products of the representations above. In \cite{nt}, Nazarov and Tarasov introduce the class of \emph{tame} representations i.e. representations of the form $\bigotimes_{i=1}^k V_{\lambda_i \setminus \mu_i}(z_i)$ such that  $z_i - z_j \not \in \bz$ for all $i \ne j$. This is the class of irreducible representations of $Y(\fgl_n)$ such that the Cartan subalgebra $H\subset Y(\fgl_n)$ acts without multiplicities. So it is natural to expect similar properties for the action of Bethe subalgebras on this class of representations of $Y(\fgl_n)$. The eigenbasis for the Cartan subalgebra $H\subset Y(\fgl_n)$, known as the Gelfand-Tsetlin basis, is naturally indexed by semistandard skew Young tableaux and is described explicitly. The eigenbasis for a general Bethe subalgebra $B(X)$ is then a deformation of the Gelfand-Tsetlin basis (being itself much less explicit).    
 
\subsection{Bethe ansatz, cyclic vector and simplicity of spectrum}
The main problem in the XXX integrable system is the diagonalization of the subalgebras $B(C)$ in the corresponding representation of the Yangian. The standard approach is the \emph{algebraic Bethe ansatz} which gives an explicit formula the eigenvectors depending on auxiliary parameters satisfying some system of algebraic equations called \emph{Bethe ansatz equations}, see for example \cite{kr}. 

The questions we address in the present paper are closely related to the \emph{completeness} of the algebraic Bethe ansatz, i.e. to the problem whether the eigenvectors obtained by Bethe ansatz form a basis in $V$. This problem is extensively studied for many years, see e.g. \cite{mv03,mtv07,mtv09,mtv,vt}. As the first step towards the solution of this problem, it is necessary that the joint eigenvalues have no multiplicities. The latter is satisfied if and only if the following two conditions hold: first, there is a \emph{cyclic vector} for the Bethe subalgebra in $V$ (i.e. $v\in V$ such that $B(C)v=V$) and, second, the algebra $B(C)$ acts on $V$ semisimply.    

Let $X \in \overline{M_{0,n+2}}$ and consider the Bethe subagebra $B(X)$. 
We call $v \in V$ cyclic with respect to $B(X)$ if  $B(X) \cdot v = V$. 
\begin{conj}
$B(X)$ has a cyclic vector in any tame representation of $Y(\fgl_n)$ for all $X \in \overline{M_{0,n+2}}$.
\end{conj}
We will discuss some facts supporting this conjecture in Section \ref{cyclic}. 

In fact, it is easy to see that the Conjecture is true for \emph{generic} $X,z_1,\ldots,z_k$. Indeed, consider the parameter space $\overline{M_{0,n+2}} \times \bc^n$. The condition that $B(X)$ acts with a cyclic vector on $\bigotimes_{i=1}^k V_i(z_i)$ determines a Zariski open subset of $\overline{M_{0,n+2}} \times \bc^n$, therefore once we have a single point $(X, z_1, \ldots, z_n) \in \overline{M_{0,n+2}} \times \bc^n$ such that $B(X)$ acts with a cyclic vector on $\bigotimes_{i=1}^k V_i(z_i)$  we automatically have the same property for generic $(X, z_1, \ldots, z_n)$. On the other hand, according to \cite{nt} the Gelfand-Tsetlin subalgebra of $Y(\fgl_n)$ (which is a particular case of $B(X)$) acts without multiplicities on any tame representation, so has a cyclic vector -- hence this Zariski-open subset is non-empty. The problem with this argument is that it does not give any representation such that $B(X)$ acts cyclicly for \emph{all} $X\in \overline{M_{0,n+2}}$. Now let us formulate the first main theorem of the present paper.

\medskip

\noindent {\bf Theorem A.} \emph{ There is a Zariski open dense subset of $I\subset\bc^n$
such that $B(X)$ has a cyclic vector in $V$ for all $X \in \overline{M_{0,n+2}}$ and $(z_1,\ldots, z_k)\in I$.  }

\medskip







Particularly, in a generic tame representation in the sense of \cite{nt} every Bethe subalgebra $B(X)$ with $X\in \overline{M_{0,n+2}}$ acts with a cyclic vector. This allows to study the joint spectrum of $B(X)$ in a given tame representation as a finite covering of $\overline{M_{0,n+2}}$ and reformulate some properties of this spectrum in terms of geometry of Deligne-Mumford compactifications.
We will discuss other useful consequences of Theorem A in Sections \ref{cyclic}, \ref{6}, see also \ref{simple}.

\subsection{Quantum cohomology of quiver varieties.}

In \cite{mo} Maulik and Okounkov describe an action of $Y(\fgl_n)$ on the localized equivariant cohomology of type A quiver varieties such that some elements from Bethe subalgebras with $C \in T^{reg}$ act as operators of quantum multiplication by some cohomology classes (with $C$ being the quantum parameter). According to a conjecture of \cite{mo} the quantum cohomology ring coincides with the image of the corresponding Bethe subalgebra.  

The $Y(\fgl_n)$-modules arising in this construction have the form $\bigotimes_{i=1}^k W_i(z_i)$ with certain parameters $z_i, i = 1, \ldots, k$, where all $W_i$'s are fundamental representations of $\fgl_n$. In particular, this conjecture implies that $B(C)$ acts with a cyclic vector on any tensor product of fundamental $Y(\fgl_n)$-modules with generic evaluation parameters $z_i$, since the unity class is always cyclic for the quantum cohomology ring. Theorem~A gives some evidence for this conjecture.

Also in \cite{varn} it is proved that the image of Gelfand-Tsetlin subalgebra $H\subset Y(\fgl_n)$, which has the form $B(X)$ for some $X \in \overline{M_{0,n+2}}$ is generated bythe operators of (classical) multiplication by equivariant cohomology classes of the quivier variety. Note that $H$ corresponds to so-called caterpilar curves, i.e. one of the most degenerate Bethe subalgebras corresponding to $0$-dimensional strata of $\overline{M_{0,n+2}}$, so this particular degenerate curve corresponds to the classical limit of the quantum cohomology ring. In this perspective it is interesting to understand the geometric nature of other types of limit Bethe subalgebras.

\subsection{Plan of the proof.}
The main idea of the proof is to reduce the eigenproblem for $XXX$ model to that for the quantum \emph{shift of argument subalgebras}, cf. \cite{r06}. 
The simplest example is as follows. Suppose that $k=1$, i.e. we study the question about a cyclic vector in the representation of the form $V(z)$, where $V$ is any irreducible representation of $\fgl_n$. It is well-known, see \cite{nazol}, that $\ev_z(B(C)) = \mathcal{A}_{C^{-1}}$ for any $C \in T^{reg}$ and $z \in \bc$ with $\mathcal{A}_{C^{-1}}\subset U(\fgl_n)$ being the shift of argument subalgebra defined in \cite{r06}. On the other hand, for the shift of argument subalgebras it is known (see \cite{ffr}) that any $\mathcal{A}_C$ with $C \in T^{reg}$ acts with a cyclic vector (moreover, with simple spectrum) on any irreducible representation of $\fgl_n$.

In the general situation,
the set of values of the parameters $z_i$ such that Theorem~A is \emph{not} satisfied is Zariski-closed, since the variety $\overline{M_{0,n+2}}$ is complete. So it is sufficient to prove that Theorem~A is satisfied at least for some values of the $z_i$'s. The most convenient choice of such particular value of the evaluation parameters is $z_i\in\bz$ with $|z_i - z_j| \gg0$ since then $\bigotimes_{i=1}^k V_{\lambda_i \setminus \mu_i}(z_i)$ is in fact a skew representation of the Yangian. This reduces Theorem~A to the case of one tensor factor which is a multiplicity space $M_{\lambda\mu}$. The image of $B(X)$ under the centralizer construction map $\eta_k: Y(\fgl_n)\to U(\fgl_{n+k})^{\fgl_k}$ turns to coincide with certain limit shift of argument subalgebra in $U(\fgl_{n+k})$, hence we reduce Theorem~A to the similar question on shift of argument subalgebras. Finally, according to \cite{hkrw}, the corresponding 
shift of argument subalgebra acts with a cyclic vector on any representation of $U(\fgl_{n+k})^{\fgl_k}$ in the multiplicity space $M_{\lambda\mu}$ for any $\lambda$ and $\mu$.


\subsection{Simple spectrum property and monodromy conjectures}
\label{simple}
Theorem A implies that once $B(C)$ acts semisimply, it has simple spectrum (i.e. the joint eigenvalues have no multiplicities). The usual sufficient condition for this is the existence of a Hermitian scalar product such that $B(C)^+=B(C)$ i.e. all elements of $B(C)$ act by normal operators. We give sufficient conditions on the representation of the Yangian guaranteeing that such scalar product exists provided that $C$ belongs either to the closure of the set of regular elements of the compact real torus $T_{comp}\subset T$ or to that of the split real torus $T_{split}\subset T$. 

The case of the compact torus goes back to Kirillov and Reshetikhin \cite{kr}. Recall that Kirillov-Reshetikhin module is an irreducible $\fgl_n$-module corresponding to a rectangular Young diagram (equivalently, it is the highest irreducible component  of a symmetric power of a fundamental representation of $\fgl_n$). The second main result of this paper is the following

\medskip

\noindent {\bf Theorem B.} 
\emph{Suppose that all $V_i$'s are Kirillov-Reshetikhin modules. Let $l_i \times r_i$ be the size of the corresponding Young diagram. Suppose that  
$z_i = \frac{l_i}{2} - \frac{r_i}{2} + i x_i $, where $x_i \in \br$.
Then, for $(x_1, \ldots, x_k)  \in \br^k$ from Zariski open subset, $B(X)$ has simple spectrum on $\bigotimes_{i=1}^k V_i(z_i)$ for all $X \in \overline{M_{0,n+2}^{comp}}$.}

\medskip

Theorem B allows to regard the set of eigenlines for $B(X)$ in $\bigotimes_{i=1}^k V_i(z_i)$ as an \emph{unramified} covering of the space $\overline{M_{0,n+2}^{comp}}$. In particular, we get the monodromy action of the fundamental group $\pi_1(\overline{M_{0,n+2}^{comp}})$ (which is natural to call (pure) \emph{affine cactus group}) on the spectrum of Bethe subalgebras.  Moreover, it is possible to define the structure of a Kirillov-Reshetikhin crystal on this spectrum, following the strategy of \cite{hkrw} (see \cite{kmr} and Section~\ref{further} for the details). We expect that the action of the affine cactus group on this set is given by partial Schutzenberger involutions on the KR-crystal. 

The closure of the set of regular points of the split real torus $T_{split}$ in $\overline{M_{0,n+2}}$ is the real form  $\overline{M_{0,n+2}^{split}}$. Our third main result is

\medskip

\noindent {\bf Theorem C.} 
\emph{Let $V_i, i = 1, \ldots, k$ be a set of skew representations of $Y(\fgl_n)$. Then, for $(x_1, \ldots, x_k)  \in \br^k$ from Zariski open subset, $B(X)$ has simple spectrum on $\bigotimes_{i=1}^k V_i(x_i)$ for all $X \in \overline{M_{0,n+2}^{split}}$.}

\medskip

Similarly to the case of the compact real form, this gives an action of the usual cactus group $\pi_1(\overline{M_{0,n+2}^{split}})$ on the spectrum of a Bethe algebra. Specializing the parameter of the Bethe algebra to the \emph{caterpillar} point of $\overline{M_{0,n+2}^{split}}$ we get an action of the cactus group on the  Gelfand-Tsetlin basis in the tensor product of skew representations. The latter is indexed by collections of semistandard skew Young tableaux, and we conjecture that the action of the cactus group is given by Bender-Knuth involutions, similarly to the construction of Chmutov, Glick and Pylyavskyy \cite{cgp}. 

\subsection{The paper is organized as follows} In Section 2 we give some preliminaries on Yangians. In Section 3 we discuss Deligne-Mumford compactification and its real forms. In Section 4 we define the families of commutative subalgebras in $U(\fgl_n)$ and in $Y(\fgl_n)$ and describe relations between them. In Section 5 we prove Theorem A. In section 6 we prove Theorems B and C. In Section 7 we discuss further directions of the research.

\subsection{Acknowledgements}
We thank Vasily Krylov for many valuable comments on the early draft of the paper. We thank Joel Kamnitzer and Vitaly Tasrasov for extremely useful discussions. The study has been funded within the framework of the HSE University Basic Research Program. The work of A.I. and L.R. on Theorem C was supported by the RSF grant 19-11-00056. I.M.-G. was supported by RFBR grant number 19-31-90124.  L.R. was partially supported by the Foundation for the Advancement of Theoretical Physics and Mathematics
 ``BASIS". A.I. is a Young Russian Mathematics award winner and would like to thank its sponsors and jury.

\section{Yangian and universal enveloping algebra}
\label{def}
\subsection{Yangian of $\fgl_n$ and $\fsl_n$}
\begin{definition}
{\itshape Yangian of $\fgl_n$} is a complex unital associative algebra with countably many generators $t_{ij}^{(1)}, t_{ij}^{(2)}, \ldots $ where $1 \leqslant i,j \leqslant n$, and the defining relations
$$ [t_{ij}^{(r+1)},t_{kl}^{(s)}] - [t_{ij}^{(r)},t_{kl}^{(s+1)}] = t_{kj}^{(r)}t_{il}^{(s)} - t_{kj}^{(s)}t_{il}^{(r)}, $$
where $r,s \geqslant 0$ and $t_{ij}^{(0)} = \delta_{ij}$. This algebra is denoted by $Y(\fgl_n)$.
\end{definition}

There is a filtration on the Yangian determined by $\deg t_{ij}^{(r)}=r$. The $t_{ij}^{(r)}$'s are PBW generators with respect to this filtration. 

It is convenient to consider the formal series
$$t_{ij}(u) = \delta_{ij} + t_{ij}^{(1)} u^{-1} + t_{ij}^{(2)} u^{-2} + \ldots \in Y(\fgl_n)[[u^{-1}]].$$

We denote by $T(u)$ the matrix whose $ij$-entry is $t_{ij}(u)$.
We regard this matrix as the following element of $Y(\fgl_n)[[u^{-1}]] \otimes \rm{End} \ \mathbb{C}^n$:
$$T(u) = \sum_{i,j=1}^n t_{ij}(u) \otimes e_{ij},$$
where $e_{ij}$ stands for the standard matrix units.

Consider the algebra
$$Y(\fgl_n)[[u^{-1}]] \otimes (\rm{End} \ \mathbb{C}^n)^{\otimes n}.$$
For any $a \in \{1, \ldots ,n\}$ there is an embedding $$
i_a: Y(\fgl_n)[[u^{-1}]] \otimes \rm{End} \ \mathbb{C}^n\to Y(\fgl_n)[[u^{-1}]] \otimes (\rm{End} \ \mathbb{C}^n)^{\otimes n}
$$
which is an identity on $Y(\fgl_n)[[u^{-1}]]$ and embeds $\rm{End} \ \mathbb{C}^n$ as the $a$-th tensor factor in $(\rm{End} \ \mathbb{C}^n)^{\otimes n}$. Denote by $T_a(u)$ the image of $T(u)$ under this embedding.

Consider $R(u) = 1 - Pu^{-1} \in (\End \bc^n)^{\otimes 2}$, where $P = \sum_{i,j=1}^n e_{ij} \otimes e_{ji}$.
Then it is well-known that one can rewrite defining relations for $Y(\fgl_n)$ as follows:
$$R(u-v) T_1(u) T_2(v) = T_2(u) T_1(v) R(u-v).$$

$Y(\fgl_n)$ is a Hopf algebra with the comultiplication:
$$\Delta: Y(\fgl_n) \to Y(\fgl_n) \otimes Y(\fgl_n), t_{ij}(u) \mapsto \sum_{k=1}^n t_{ik}(u) \otimes t_{kj}(u)$$
and the antipode
$$S: T(u) \to T^{-1}(u).$$

For any $1 \leq a_1 \leq \ldots \leq a_k \leq n$ and $1 \leq b_1 \leq \ldots \leq b_k \leq n$ by definition put
$$t^{a_1 \ldots a_k}_{b_1 \ldots b_k}(u) = \sum_{\sigma \in S_k} (-1)^\sigma \, \cdot \, t_{a_{\sigma(1)}b_1}(u) \ldots t_{a_{\sigma(k)}b_k}(u-k+1) = \sum_{\sigma \in S_k} (-1)^\sigma \, \cdot \, t_{a_1 b_{\sigma(1)}}(u-k+1) \ldots t_{a_k b_{\sigma(k)}}(u). $$
We call series $t^{a_1 \ldots a_k}_{b_1 \ldots b_k}(u)$ {\it quantum minors}.

We also use Yangian $Y(\fsl_n)$ for the Lie algebra $\fsl_n$.
\begin{definition}
Yangian $Y(\fsl_n)$ is the subalgebra of $Y(\fgl_n)$ which consists of all elements stable under the automorphisms $T(u) \to f(u) T(u)$ for all $f(u) \in 1 + u^{-1}\bc[[u^{-1}]]$.
\end{definition}
\begin{proposition}(\cite[Theorem 1.8.2, Proposition 1.84.]{M})
\label{tensor2}
\begin{enumerate}
\item $Y(\fgl_n) = Y(\fsl_n) \otimes_{\bc} Z(Y(\fgl_n))$, where $Z(Y(\fgl_n))$ is the center of $Y(\fgl_n)$;
\item $Y(\fsl_n)$ is a Hopf subalgebra of $Y(\fgl_n)$.
\end{enumerate}
\end{proposition}

\begin{remark}
This means that one can also realize $Y(\fsl_n)$ as the quotient of $Y(\fgl_n)$ by a maximal ideal of the center $Z(Y(\fgl_n))$, see \cite[Corollary 1.8.3]{M}.
\end{remark}

\subsection{Evaluation homomorphism and evaluation modules}
Let $z \in \bc$. It is well-known that the map
$$\tau_z: Y(\fgl_n) \to Y(\fgl_n), T(u) \mapsto T(u-z)$$
is an automorphism of $Y(\fgl_n)$.

Let $E = \sum\limits_{i,j=1}^n E_{ij}\otimes e_{ij}=\left( e_{ij}\right)_{i,j = 1, \ldots, n}\in U(\fgl_n)\otimes Mat_n$ be the $n\times n$-matrix with the $(i,j)$-th coefficient being $E_{ij}\in U(\fgl_n)$ (here $E_{ij} \in \fgl_n$ is the matrix with 1 on $(i,j)$-entry and zeros on other entries).
It is well known that the map 
$$\ev: Y(\fgl_n) \to U(\fgl_n), T(u) \mapsto 1 + E u^{-1}$$ is a homomorphism of algebras.
By definition put
$$\ev_z := \ev \circ \tau_z$$

Any $\fgl_n$-module $V$ can be regarded as a $Y(\fgl_n)$-module by means of $\ev = \ev_0$. Moreover, using $\ev_z$ one can produce a $1$-parametric family $V(z)$ of $Y(\fgl_n)$-modules. 



We also note that the elements $t_{ij}^{(1)}, 1 \leq i,j \leq n$ generate a copy of $U(\fgl_n)$ in $Y(\fgl_n)$. This allows to consider any representation of $Y(\fgl_n)$ as a representation of $U(\fgl_n)$, hence as a representation of the corresponding connected simply connected group $\tilde{GL_n}$.

\subsection{Some homomorphisms between Yangians}
Let us define two different embedding of $Y(\fgl_n)$ to $Y(\fgl_{n+k})$:
$$i_k: Y(\fgl_n) \to Y(\fgl_{n+k}); \ t^{(r)}_{ij} \mapsto t^{(r)}_{ij}$$
$$\varphi_k: Y(\fgl_n) \to Y(\fgl_{n+k}); \ t^{(r)}_{ij} \mapsto t^{(r)}_{k+i,k+j}$$

By definition, put
$$\omega_n: Y(\fgl_n) \to Y(\fgl_n); \ T(u) \mapsto (T(-u-n))^{-1}.$$
It is well-known that $\omega_n$ is an involutive automorphism of $Y(\fgl_n)$.
We define a homomorphism
$$\psi_k = \omega_{n+k} \circ \varphi_k \circ \omega_n: Y(\fgl_n) \to Y(\fgl_{n+k}).$$
Note that $\psi_k$ is injective. 

\begin{proposition} \cite[Lemma 2.12]{ir2}
\label{commuteyang} The homomorphisms $i_{k}$ and $\psi_{n}$ define an embedding
$i_{k}\otimes\psi_{n}: Y(\fgl_{n}) \otimes Y(\fgl_k) \hookrightarrow Y(\fgl_{n+k})$.
\end{proposition}

\subsection{Centralizer construction}
Consider the map
\begin{equation*}
\Phi_k: Y(\fgl_n) \to U(\fgl_{n+k}) \ \text{given by} \ \Phi_k = \ev \circ \omega_{n+k} \circ i_k.
\end{equation*}
From ~\cite[Proposition 8.4.2]{M} it follows that ${\rm Im} \, \Phi_k \subset U(\fgl_{n+k})^{\fgl_k}$. Here we use an embedding $$\mathfrak{gl}_k \to \mathfrak{gl}_{n+k}, \ E_{ij} \to E_{i+n, j+n}.$$ 

Let $A_0 = \mathbb{C}[\mathcal{E}_1, \mathcal{E}_2, \ldots]$ be the polynomial algebra of infinite many variables. Define a grading on $A_0$ by setting $\deg \mathcal{E}_i = i$. For any $k$ we have a surjective homomorphism
$$z_k: A_0 \to Z(U(\fgl_{n+k}));\ \mathcal{E}_i \to \mathcal{E}_i^{(n+k)}.$$
where $\mathcal{E}_i^{(n+k)}$, $i=1,2,3,\ldots$ are the following generators of $Z(U(\fgl_{n+k}))$ of degree $i$, see \cite[Section 8.2]{M}:
$$ 1+\sum\limits_{i=1}^{n+k}\mathcal{E}_iu^{-i}=\ev( \qdet  T(u))
$$

Consider the algebra $Y(\fgl_n) \otimes A_0$. This algebra has a well-defined ascending filtration given by
$$\deg a \otimes b = \deg a + \deg b$$

For any $k \geqslant 0$ we define homomorphisms of filtered algebras
$$\eta_k: Y(\fgl_n) \otimes A_0 \to U(\fgl_{n+k})^{\fgl_k}; \ a \otimes b \to \Phi_k(a) \cdot z_k(b)$$

These homomorphisms are known to be surjective. Denote by  $(Y(\fgl_n) \otimes A_0)_N$ the $N$-th filtered component, i.e. the vector space of the elements of degree not greater than $N$. From \cite[Theorem 8.4.3]{M} we have:

\begin{theorem}
\label{qasym}
The sequence $\{\eta_k\}$ is an asymptotic isomorphism. This means that for any $N$ there exists $K$ such that for any $k>K$ the restriction of $\eta_k$ to the $N$-th filtered component $(Y(\fgl_n) \otimes A_0)_N$ is an isomorphism of vector spaces $(Y(\fgl_n) \otimes A_0)_N \simeq U(\fgl_{n+k})^{\fgl_k}_N$.
\end{theorem}

\subsection{Skew representations of $Y(\fgl_n)$}
Let $V_{\lambda}$ be the finite-dimensional irreducible representation of $U(\fgl_{n+k})$ of the highest weight $(\lambda_1, \ldots, \lambda_{n+k})$. Consider the restriction of $V_{\lambda}$ to $\fgl_k$:
$V_{\lambda} = \sum_{\mu} M_{\lambda \mu} \otimes V_{\mu}$, where $M_{\lambda \mu} : = \Hom(V_{\mu}, V_{\lambda})$ is the multiplicity space.
It is well-known that multiplicity space is an irreducible representation of  $U(\fgl_{n+k})^{\fgl_k}$.
Restriction of $\eta_k \circ \tau_z$ gives $M_{\lambda \mu}$ a structure of representation of $Y(\fgl_n)$. We denote this representation by $V_{\lambda \setminus \mu}(z)$ and call {\em skew representation} of $Y(\fgl_n)$.

According to the Jacobson density theorem, any multiplicity space $M_{\lambda \mu}$ is an irreducible $U(\fgl_{n+k})^{\fgl_k}$-module. Since $\eta_k$ is surjective we have the following 
\begin{proposition}\cite[Section 6]{M}
Any skew representation of $Y(\fgl_n)$ is irreducible.
\end{proposition}


\section{$\overline{M}_{0,n+2}$ and its real forms.}
\label{compact}

\subsection{Deligne-Mumford compactification.}
Let $\overline{M_{0,n+2}}$ denote the Deligne-Mumford space of stable rational curves with $n+2$ marked points. The points of $\overline{M_{0,n+2}}$ are isomorphism classes of curves of genus $0$, with $n+2$ ordered marked points and possibly with nodes, such that each component has at least $3$ distinguished points (either marked points or nodes). One can represent the combinatorial type of such a curve as a tree with $n+2$ leaves with inner vertices representing irreducible components of the corresponding curve, inner edges corresponding to the nodes and the leaves corresponding to the marked points. Informally, the topology of $\overline{M_{0,n+2}}$ is determined by the following rule: when some $k$ of the distinguished points (marked or nodes) of the same component collide, they form a new component with $k+1$ distinguished points (the new one is the intersection point with the old component). In particular, the tree describing the combinatorial type of the less degenerate curve is obtained from the tree corresponding to the more degenerate one by contracting an edge.

The space $\overline{M_{0,n+2}}$ is a smooth algebraic variety. It can be regarded as a compactification of the configuration space $M_{0,n+2}$ of ordered $(n+2)$-tuples $(z_0,z_1,\ldots,z_{n+1})$ of pairwise distinct points on $\bc \bp^1$ modulo the automorphism group $PGL_2(\bc)$. Since the group $PGL_2(\bc)$ acts transitively on triples of distinct points, we can fix the $0$-th and the $(n+1)$-th points to be $0\in\bc\bp^1$ and $\infty\in\bc\bp^1$, respectively. Then the space $M_{0,n+2}$ gets identified with the quotient $\{(z_1,\ldots,z_n)\in \bc^{*n}\ |\ z_i\ne z_j\}/ \bc^*$, and the group $\bc^*$ acts by dilations. Under this identification of $M_{0,n+2}$, the space $\overline{M_{0,n+2}}$ is the multiplicative version of the De Concini-Procesi wonderful closure of the complement to a hyperplane arrangement introduced by De Concini and Giaffi in \cite{dcg}, i.e. the wonderful closure of the complement to the arrangement of root subtori in the maximal torus of the adjoint group on the toric variety $X$ attached to the fan formed by the root arrangement in the weight lattice of $T=\bc^{*n}/\bc^*$. For $G=PGL_n $ we have $X=\overline{L_{n+2}}$ is the \emph{Losev-Manin moduli space}, \cite{LM}. The following statement is well-known, see \cite{K}:

\begin{proposition} $\overline{M_{0,n+2}}$ is the iterated blow-up of the subspaces of the form $\{z_{i_1}=z_{i_2}=\ldots=z_{i_k}\}$ in $\overline{L_{n+2}}$.
\end{proposition}

\begin{proof}
According to Kapranov \cite{
K}, $\overline{M_{0,n+2}}$ is an iterated blow up of the projective space $\mathbb{P}(\bc^n)$ at the subspaces of the form $\{0=z_{i_1}=z_{i_2}=\ldots=z_{i_k}\}$ and $\{z_{i_1}=z_{i_2}=\ldots=z_{i_k}\}$. The projective space $\mathbb{P}(\bc^n)$ is a toric $T$-variety corresponding to the fan formed by the cones $C_i=\{x_i\ge x_j\ |\ \forall j\ne i\}$ for all $i=1,\ldots,n$. The fan formed by the root arrangement is a subdivision of the fan formed by the cones $C_i$, and the additional faces correspond to the subvarieties of the form $\{0=z_{i_1}=z_{i_2}=\ldots=z_{i_k}\}$ on $\mathbb{P}(\bc^n)$. So the variety $\overline{L_{n+2}}$ is the blow-up of $\mathbb{P}(\bc^n)$ at the subspaces of the form $\{0=z_{i_1}=z_{i_2}=\ldots=z_{i_k}\}$. Hence $\overline{M_{0,n+2}}$ is the blow-up of the subvarieties of the form $\{z_{i_1}=z_{i_2}=\ldots=z_{i_k}\}$ on $\overline{L_{n+2}}$.
\end{proof}

\begin{remark}
It is well-known that $\overline{L_{n+2}}$ is also the closure of the maximal torus $T$ in the De Concini-Procesi wonderful closure $\overline{G}$ of the adjoint group $G=PGL_n$. See also \ref{limits}. 
\end{remark}

\subsection{Stratification of $\overline{M_{0,n+2}}$}
The Deligne-Mumford variety $\overline{M_{0,n+2}}$ is stratified by the combinatorial type of the curve with marked points. Namely, the strata are indexed by trees whose leaves correspond to the marked points $0,z_1,\ldots,z_n,\infty$ where the inner vertices represent the connected components and the edges correspond to the nodal points. Clearly, the dimension of the stratum is the sum of the degree minus $3$ over all inner vertices. The natural partial order on the strata is given by contracting edges of the trees. Clearly, the closure of each stratum is the direct product $\prod\limits_{i\in I} \overline{M_{0,k_i}}$ where $I$ is the set of inner vertices of the tree and the $i$-th inner vertex is $k_i$-valent. In particular, starting from a point in the open stratum, one can approach any boundary point of $\overline{M_{0,n+2}}$ just by successively approaching a generic point in a codimension $1$ strata for subsequent $\overline{M_{0,k_i}}$'s. 

There are strata of the codimension $1$ parameterize $2$-component curves (eqivalently, trees with exactly $2$ inner vertices). We distinguish the case when the points $0$ and $\infty$ belong to different components and call them \emph{strata of the first type}. The corresponding curves have the form (these strata are already visible on the Losev-Manin space):
\begin{center}
\includegraphics[scale=0.5]{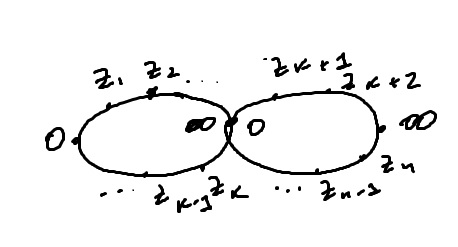}
\end{center}

All other codimension $1$ strata are called \emph{strata of the second type} and have the form
\begin{center}
\includegraphics[scale=0.5]{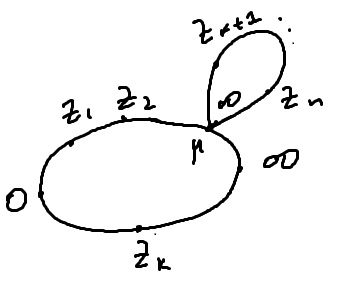}
\end{center}

Note that one can approach any boundary point by first passing successively to strata of the first type (this corresponds to approaching boundary points in $\overline{L_{n+2}}$) and next to strata of the second type in appropriate $\overline{M_{0,k_i}}$ (this corresponds to approaching a boundary point in a fiber of $\overline{M_{0,n+2}}$ over $\overline{L_{n+2}}$).



\subsection{Real forms} 
\label{inv}
In \cite{ceyhan}, \"Ozg\"ur Ceyhan studies real forms of the Deligne-Mumford space which preserve the stratification. To any involution (including identity) $\sigma\in S_{n+2}$ one can assign an involutive automorphism of $\overline{M_{0,n+2}}$ which 
permutes the marked points. Composing it with the complex conjugation $z_i\mapsto \bar{z_i}$ we get a complex antilinear automorphism $s:\overline{M_{0,n+2}}\to\overline{M_{0,n+2}}$. According to \cite{ceyhan}, these are all complex antilinear automorphisms preserving the stratification of $\overline{M_{0,n+2}}$. More precisely, to an involution $\sigma$ being the product of $k$ independent transpositions corresponds the real form $\overline{M_{2k,l}}$ (with $l+2k=n+2$) such that generic points of $M_{2k,l}$ have the form $(z_1,\ldots,z_k,\bar{z}_1,\ldots,\bar{z}_k,u_1,\ldots,u_l)$ where $u_i\in\br$.

Let $\overline{M_{0,n+2}}^s$ be the real locus of the Deligne-Mumford space with respect to the real structure $s$.  Note that if $s_1$ and $s_2$ are two conjugated antiholomorphic involutions then $\overline{M_{0,n+2}}^{s_1}$ and $\overline{M_{0,n+2}}^{s_2}$ are diffeomorphic. 
The real forms of the algebraic group $PGL_n$ have very similar description. Let $G_r$ be a real form of $PGL_n$ and $T_r \subset G_r$ be a maximal torus which contains a maximal split torus. We can assume that $T_r \subset T$. Recall that we have a map from $T$ to $M_{0,n+2}$ so can regard the torus $T_r$ as a subset of $M_{0,n+2}$. One can find an antiholomorphic involution $s$ of $\overline{M_{0,n+2}}$ such that $T_r = (M_{0,n+2})^{s}$ hence an open subset of $\overline{M_{0,n+2}}^{s}$. This means that the closure of $T_r$ in $\overline{M_{0,n+2}}$ coincides with $\overline{M_{0,n+2}}^{s}$ which defines a map from real forms of $PGL_n$ to real forms of $\overline{M_{0,n+2}}$. We have the following

\begin{proposition}
The above map from the equivalence classes of real forms of the complex algebraic group $PGL_n$ to the real forms of $\overline{M_{0,n+2}}$ is surjective. This map is one-to-one for odd $n$.
\end{proposition}

\begin{proof}
Indeed, for $k>0$, $\overline{M_{2k,l}}$ corresponds to $U_{k+l-1,k-1}$ and the ``split'' real form $k=0$ corresponds to $GL_n(\br)$. 
\end{proof}


We will be interested in two simplest cases, namely
\begin{enumerate}
    \item $\sigma_0 = e\in S_{n+2}$. Then the set of $\br$-points of the corresponding real form is the closure of the set regular points of the \emph{split} form of the maximal torus $T_{split}\subset PGL_n(\br)$, i.e. the real locus is $\overline{M_{0,n+2}^{split}}=\overline{\{(z_1,\ldots,z_n)\ |\ z_i\ne z_j, z_i\in\br\}/\br^\times}$. 
    This is the usual real form of $\overline{M_{0,n+2}}$ which plays the crucial role in \emph{coboundary} monoidal categories analogous to the role played by the space of configurations of points on the complex line in braided monoidal categories, see \cite{hk,hkrw} for details.
    \item $\sigma_1 = (1,n+2)\in S_{n+2}$. The corresponding real locus of $\overline{M_{0,n+2}}$ is the closure of the set of regular points of the \emph{compact} real form of the maximal torus $T_{comp}\subset PU_u$, i.e. the real locus is $\overline{M_{0,n+2}^{comp}}=\overline{\{(z_1,\ldots,z_n)\ |\ z_i\ne z_j, |z_i|=1\}/U_1}$.  Namely, the points of $\overline{M_{0,n+2}^{comp}}$ are curves where all marked points are on the unit circle, and the points $0, \infty$ belongs to the same component. In particular, this means that the projection of this real locus to the Losev-Manin space is contained in the open $T$-orbit, i.e. there are no boundary strata involving the degeneration of the first type.
\end{enumerate}

For the future reference, we write down here the action of $\sigma_1$ on codimension~$1$ strata of the first and of the second type of $\overline{M_{0,n+2}}$:
\begin{center}
\includegraphics[scale=0.5]{lim1.jpg} $\mapsto$ \includegraphics[scale=0.63]{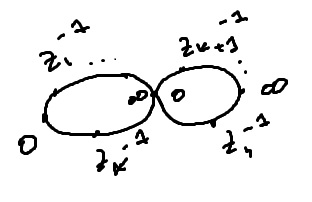}
\end{center}
\begin{center}
\includegraphics[scale=0.5]{lim2.jpg} $\mapsto$ \includegraphics[scale=0.63]{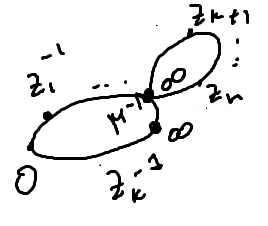}
\end{center}

\section{Bethe and shift of argument subalgebras}
\subsection{Bethe subalgebras}

The symmetric group $S_n$ acts on $Y(\fgl_n)[[u^{-1}]] \otimes (\rm{End} \ \mathbb{C}^n)^{\otimes n}$ by permuting the tensor factors. This action factors through the embedding $S_n\hookrightarrow (\rm{End} \ \mathbb{C}^n)^{\otimes n}$ hence the group algebra $\bc[S_n]$ is a subalgebra of $Y(\fgl_n)[[u^{-1}]] \otimes (\End \ \mathbb{C}^n)^{\otimes n}$. Let $S_m$ be the subgroup of $S_n$ permuting the first $m$ tensor factors. Denote by $A_m$ the antisymmetrizer $$\sum_{\sigma \in S_m} (-1)^{\sigma}\sigma \in \mathbb{C}[S_m]\subset Y(\fgl_n)[[u^{-1}]] \otimes (\rm{End} \ \mathbb{C}^n)^{\otimes n}.$$
Suppose that $C \in GL_n$. For any $a \in \{1, \ldots ,n\}$ denote by $C_a$ the element $i_a(1\otimes C)\in Y(\fgl_n)[[u^{-1}]] \otimes (\operatorname{End} \mathbb{C}^n)^{\otimes n}$. For any $1 \leqslant p \leqslant n$ introduce the series with coefficients in $Y(\fgl_n)$ by
$$\tau_p(u,C) =  \tr A_p  C_1 \ldots C_p T_1(u) \ldots T_p(u-p+1),$$
where we take the trace over all $p$ copies of $ \rm{End} \ \mathbb{C}^n$.

\begin{definition}
\emph{Bethe subalgebra} $B(C)\subset Y(\fgl_n)$ is generated by all coefficients of the series $\tau_p(u,C)$ for $p=1,\ldots,n$.
\end{definition}

It follows from the definition that $B(C) = B(a \cdot C)$ for any $a \in \bc$. We fix a maximal torus $T\subset GL_n$ (i.e. the subgroup of diagonal matrices) and denote by $T^{reg}$ the set of regular elements of $T$. Denote by $GL_n^{reg}$ the set of regular elements of the Lie group $GL_n$. The following Proposition summarize known algebraic properties of Bethe subalgebras, see e.g. \cite{nazol}, \cite{ir3}, \cite{ilin}.

\begin{theorem}[Properties of Bethe subalgebras]
\label{betheprop}
\begin{enumerate}
    \item For any $C \in GL_n$, the subalgebra $B(C)$ is commutative.
    \item For $C \in T^{reg}$, the subalgebra $B(C)$ is a maximal commutative subalgebra. 
    \item For any $C \in GL_n^{reg}$ the subalgebra $B(C)$ is freely generated by the coefficients of $\tau_p(u,C)$ with $p=1,\ldots,n$.
    \item Let $C \in T^{reg}$ and let $\tilde C$ be an arbitrary representative of $C$ in the universal cover of $GL_n$. Then the subalgebra $B(C)$ generated by all $$\tr_V \rho(\tilde C) (\rho \otimes 1)(\hat R(u)),$$ where $(\rho,V)$ ranges over all finite-dimensional representations of $Y(\fgl_n)$ and $\hat R(u)$ is the universal $R$-matrix for $Y(\fgl_n)$.
\end{enumerate}

\end{theorem}
In fact there are no doubts that the assertion (2) is true for any $C \in GL_n^{reg}$ and that (4) is true for any $C \in GL_n$ but still there is no rigorous proof in the literature. Also, one can take (4) as the definition of Bethe subalgebras in the case of any simple Lie algebra, see \cite{ilin}. 
\begin{remark}
\emph{In the literature, Bethe subalgebras are often assigned to any matrix $C \in \fgl_n$ (not necessarily invertible). One can regard these subalgebras as the ones corresponding to some points in the {\em wonderful compactification} of the group $GL_n$, see \ref{limits} for details.} 
\end{remark}

\begin{proposition}
\label{tensor}

Bethe subalgebra $B(C)$ of $Y(\fgl_n)$ is the tensor product $B'(C)\otimes ZY(\fgl_n)$ where $B'(C)$ is a commutative subalgebra in $Y(\fsl_n)$ and $ZY(\fgl_n)$ is the center of $Y(\fgl_n)$.
\end{proposition}
\begin{proof}
Let $d(u) \in ZY(\fgl_n)[[u^{-1}]]$ be the (unique) series such that $d(u)=1$ modulo $u^{-1}$ and $$\qdet T(u) = d(u) d(u-1) \ldots d(u-n+1).$$
Then coefficients of series $$d(u)^{-1} \tau_1(u,C), (d(u)^{-1}d(u-1)^{-1}) \tau_2(u,C) , \ldots, (d(u)^{-1} d(u-1)^{-1} \ldots d(u-n+2)^{-1}) \tau_{n-1}(u,C)$$ and $\tau_n(u,C) = \qdet T(u) $ are free generators of $B(C)$. All coefficients of the first $n-1$ series then belong to $Y(\fsl_n)$, and the coefficients of $\qdet T(u)$ generate the center, \cite[Theorem 1.7.5]{M}. The statement of the Proposition follows.
\end{proof}
We will call $B'(C)$ \emph{Bethe subalgebra of} $Y(\fsl_n)$ and also denote it (slightly abusing notation) by $B(C)$. Theorem \ref{betheprop} holds for Bethe subalgebras of $Y(\fsl_n)$ as well.
In the present paper we restrict ourselves by $C \in T^{reg}$, i.e. we fix maximal torus and consider the family of Bethe subalgebras parameterized by its regular points. 

\subsection{Quantum shift of argument subalgebras.}
There is a family of commutative subalgebras in $U(\fgl_n)$ closely related to Bethe subalgebras, called {\em quantum shift of argument subalgebras}, $\mA_\chi\subset U(\fgl_n)$ depending $\chi \in \fgl_n$.
One can describe these subalgebras as {\em liftings} of 
{\em shift of argument subalgebras} in $S(\fgl_n)$. Namely,
by PBW thorem, $\gr U(\fgl_n) = S(\fgl_n)$. This determines a natural Poisson structure on $S(\fgl_n)$ which is defined on the generators as $$\{x,y\} = [x,y]  \ \forall \ x,y \in \fgl_n$$. 

Recall the matrix $E=\sum\limits_{i,j=1}^n E_{ij}\otimes e_{ij}\in S(\fgl_n)\otimes Mat_n$. Let $E_{j_1, \ldots, j_i}^{j_1, \ldots, j_i}$ be the principal minor of $E$ corresponding to  $j_1, \ldots, j_i$ rows and columns of $E$. The Poisson center of $S(\fgl_n)$ is freely generated by $n$ homogeneous generators $P_1, \ldots, P_n$ of degrees $1, \ldots, n$. There are several natural choices of the generators $P_i$; in particular, one can take  either the traces of powers $P_i = \tr E^i, i = 1, \ldots, n$ or the coefficients of the characteristic polynomial $P_i = \sum_{1 \leq j_1 \leq \ldots \leq j_i \leq n} E_{j_1, \ldots, j_i}^{j_1, \ldots, j_i}$.

\begin{definition}
The subalgebra of $S(\fgl_n)$ generated by elements
$$\partial^{c_i}_{\chi} P_i, 1 \leq c_i \leq i - 1$$ is called shift of argument subalgebra and denoted by $\overline{\mA_{\chi}}$.
\end{definition}

\begin{proposition} \cite{mf}
For any $\chi \in \fgl_n$ subalgebra $\overline{\mA_{\chi}}$ is Poisson commutative.
\end{proposition}

\begin{definition}
We say that the subalgebra $\mA_{\chi} \subset U(\fgl_n)$ is a lifting of $\overline{\mA_{\chi}}$ if $\mA_{\chi}$ is commutative and $\gr \mA_{\chi} = \overline{\mA_{\chi}}$. We call such subalgebras {\em quantum shift of argument subalgebras}.
\end{definition}

\begin{proposition}(\cite{fult}, \cite{taras2}, \cite{fftl}, \cite{r06})
\begin{itemize}
\item For any $\chi \in \fgl_n$ there exists a lifting $\mA_{\chi}$ of $\overline{\mA_{\chi}}$;

\item For $\chi \in \fh^{reg}$ this lifting is unique;

\item The subalgebras $\mA_{\chi}$ with $\chi \in \fgl_n^{reg}$ are free and maximal commutative;

\end{itemize}
\end{proposition}

In the present paper we restrict ourselves by $\chi \in \fh$, i.e. we fix maximal torus $T$ and its tangent algebra $\fh$ and consider the family of shift of argument subalgebras parameterized by $\fh$.

A natural question here is how to define explicitly quantum shift of argument subalgebras? One possible way is as follows. Recall the symmetrization map
$$S(\fgl_n) \to U(\fgl_n),  x_1 \ldots x_{s} \mapsto \frac{1}{s!} \sum_{\sigma \in S_s} x_{\sigma(1)} \ldots x_{\sigma(s)}.$$
Here $x_1, \ldots, x_s \in \fgl_n$. Let $P_i, i = 1, \ldots, n$ be the coefficients of characteristic polynomial. Then for any $\chi \in \fh^{reg}$ elements $$\sym(\partial^{c_i}_{\chi} P_i), i = 1, \ldots, n, 1 \leq c_i \leq i - 1$$
 are free generators of $\hat A_{\chi}$. 
Another possible way to lift shift of argument subalgebras is use of Bethe subalgebras. Note that $GL_n$ is an open subset in the matrix algebra $Mat_n$ which is $\fgl_n$. So we will regard $GL_n$  as an open subset of its Lie algebra $\fgl_n$, respectively, we regard $T$ as an open subset of $\fh$. The following Proposition is a modified version of discussion at the end of Section 2 of \cite{nazol}.

\begin{proposition}
\label{quant}
Suppose that $C \in T$. Then
\begin{itemize}
    \item $\ev_z(B(C)) = \mA_{C^{-1}}$;
    \item If $C \in T^{reg}$, then images of the coefficients at degrees $u^{-1}, \ldots, u^{-p}$ of $\tau_p(u,C)$ are free generators of $\mA_{C^{-1}}$.
\end{itemize}
\end{proposition}
\begin{proof}
In \cite{nazol} the statement is proved for $C \in T^{reg}$.
To prove the first statement for any $C \in T$, the same argument gives the following inclusion 
$$\gr \ev_z(B(C)) \supset \overline{\mA_{C^{-1}}}.$$ 
On the other hand, $\overline{\mA_{C^{-1}}}$ is a maximal Poisson commutative subalgebra in $S(\fgl_n)^{\fz_{\fgl_n}(C)}$ (where $\fz_{\fgl_n}(C^{-1})$ is the centralizer of $C$ in $\fgl_n$, see e.g. \cite[Lemma 4.11]{ir2}), so we have an equality here. Hence $\ev_z(B(C)) = \mA_{C^{-1}}$.

\end{proof}

\subsection{Limits of shift of argument subalgebras}

There is a way to define more Poisson commutative subalgebras. Recall that for any $\chi \in \fh^{reg}$ Poincar\'e series of subalgebra $\overline{\mA_{\chi}} \subset S(\fgl_n)$ is the same. It allows us to define the regular map $\theta_i: T^{reg} \to Gr(d(i), \dim S^i(\fgl_n))$, where $S^i(\fgl_n)$ is $i$-th graded component of $S(\fgl_n)$ and $d(i) = \dim \overline{\mA_{\chi}} \cap S^i(\fgl_n)$. Consider the closures $Z_i$ of $T^{reg}$ in Grasmannians and the inverse limit $Z = \varprojlim Z_i$. Every point of $Z$ gives us some Poisson commutative subalgebra with the same Poincar\'e series. We call this subalgebras limit shift of argument subalgebras and new family of subalgebras as compactification. For detailed definitions we refer the reader to \cite[Section 6]{ir3}.

The shift of argument subalgebras in $S(\fgl_n)$ depend on a parameter $C\in T^{reg}$ and do not change under the transformations $C\mapsto aC+bE, a \in \bc^{\times}, b \in \bc$. So the parameter space for the subalgebras $\mA_C \subset S(\fgl_n)$ can be regarded as $M_{0,n+1}$, i.e a point $C = \diag(z_1, \ldots, z_n)$ corresponds to $\bc\bp^1$ with an ordered set of marked points $z_1, \ldots, z_n, \infty$. Let $\overline{M_{0,n+1}}$ be the Deligne-Mumford compactification, see Section \ref{compact}.
\begin{theorem} (\cite{shuvalov}, \cite{afv}, \cite{tarasov}, \cite{ir2})
\begin{enumerate}
  \item
  For the family of shift of argument subalgebras $Z \simeq \overline{M_{0,n+1}}$.
    \item Limit subalgebras are free polynomial algebras.
    \item Limit subalgebras are maximal commutative.
    \end{enumerate}
\end{theorem}
In \cite{ir2} there is a practical way to assign a commutative subalgebra to any $X \in \overline{M_{0,n+1}}$. We will extensively use this description so we reproduce it here.  

Let $X_\infty$ be the irreducible component of $X$ containing the marked point $\infty$. To any distinguished point $\alpha\in X_\infty$ we assign the number $k_\alpha$ of marked points on the (reducible) curve $X_\alpha$ attached to $X_\infty$ at $\alpha$. Let $C$ be the diagonal matrix with the eigenvalues $\alpha$ of multiplicity $k_\alpha$ for all distinguished points $\alpha\in X_\infty$. Then the corresponding shift of argument subalgebra $F(C)\subset S(\fgl_n)$ is centralized by the Lie subalgebra $\bigoplus\limits_{\alpha}\fgl_{k_\alpha}$ in $\fgl_n$ and contains the Poisson center $S(\bigoplus\limits_\alpha \fgl_{k_\alpha})^{\bigoplus\limits_\alpha \fgl_{k_\alpha}}$. The subalgebra corresponding to the curve $X$ is just the product of $F(C)\subset S(\fgl_n)$ and the subalgebras corresponding to $X_\alpha$ in $S(\fgl_{k_\alpha})\subset S(\fgl_n)$ for all distinguished points $\alpha\in X_\infty$ (for this we need to define the point $\infty$ on each $X_\alpha$ -- it is just the intersection with $X_\infty$).


\begin{proposition}(\cite[Proposition 4.7]{ir2})
\label{tensorprod}
The subalgebra $\overline{\mA_X}$ corresponding to a degenerate curve $X$ is the tensor product $\overline{\mA_C} \otimes_{S(\bigoplus\limits_\alpha \fgl_{k_\alpha})^{\bigoplus\limits_\alpha \fgl_{k_\alpha}}}\bigotimes\limits_\alpha \overline{\mA_{X_\alpha}}$. 

\end{proposition}
Let us describe the simplest limit subalgebras corresponding to the case when all the curves $X_\alpha$ are irreducible. 
\begin{proposition}{\cite[Lemma 4.9]{ir2}}
\label{size} 
Suppose that $n=k_1+ \ldots +k_l$, where $k_1, \ldots, k_l \in \bz_{\geq 1}$
Let $C_0=\diag(\underbrace{a_1, \ldots, a_1}_{k_1}, \ldots, \underbrace{a_l, \ldots, a_l}_{k_l})$ and $C_i=\diag(\underbrace{b_{i,1}, \ldots, b_{i,k_i}}_{k_i})$ for $i=1,\ldots,l$ such that $a_r\ne a_s$ and $b_{i,r}\ne b_{i,s}$ for $r\ne s$. 
Then the element
\begin{eqnarray*}C(t) := C_0 + t \cdot \diag(C_1, \ldots, C_l)\end{eqnarray*}
belongs to $T^{reg}$ for small $t \not = 0$.
The subalgebra $\lim_{t \to 0} \overline{\mA_{C(t)}}$ is the tensor product $$\overline{\mA_{C_0}} \otimes_{Z(S(\fgl_{k_1}\oplus \ldots \oplus \fgl_{k_l}))} (\overline{\mA_{C_1}}\otimes\ldots\otimes \overline{\mA_{C_l}}).$$ 
Here $\overline{\mA_{C_i}}$ is a subalgebra in $S(\fgl_{k_i})\subset S(\fgl_{k_1}\oplus \ldots \oplus \fgl_{k_l})$.

\end{proposition}

Note that any limit can be obtained by iterations of limit from Proposition \ref{size}, which we call simple limits.

\begin{theorem} (\cite{taras2})
\label{t2}
Let $X \in \overline{M_{0,n+1}}$.
There exists the unique lifting $\mA_X$ of the algebra $\overline{\mA_X}$. The structure of subalgebras $\mA_X$ is the same as that of $\overline{\mA_X}$ described in Propositions \ref{tensorprod} and \ref{size}.
\end{theorem}

Note that from Theorem \ref{t2} we cannot conclude that the parameter space for the closure of the corresponding family of subalgebras in $U(\fgl_n)$ is the same. The following Theorem is \cite[Theorem 10.8]{hkrw}. We give here another proof specific to $\fgl_n$ case.
\begin{theorem}
\label{clo}
The closure of the family $\mA_{\chi}, \chi \in \fh^{reg}$ is isomorphic to $\overline{M_{0,n+1}}$.
\end{theorem} 
\begin{proof}
Let $\tilde Z$ be the closure of the family $U(\fgl_n)$. It is enough to prove that the size of a limit filtered subalgebra is not bigger than that of a generic subalgebra from this family: this implies that the associated graded of any limit subalgebra is isomorphic to some limit of associated graded subalgebras, so we obtain the proper birational map $\tilde Z_i \to Z_i$ which is bijective from Theorem \ref{t2}, thus $\tilde Z \simeq Z$.
The following Lemma implies that the size of the limit subalgebras does not actually jump.
\begin{lemma}(\cite[Theorem 4]{taras2})
For any $X \in \overline{M_{0,n+1}}$ subspace $\bigoplus_{m=0}^k \overline{\mA_{X}}^{(m)}$ is a maximal commutative subspace of $S(\fgl_n)^{(m)}$.
\end{lemma}
\end{proof}

Using the description of limit subalgebras we can refine the statement about uniqueness of $\mA_{\chi}$ for $\chi \in \fh$. Note that $\overline{\mA_{\chi}}$ belongs to $S(\fgl_n)^{\fz_{\fgl_n}(\chi)}$.
\begin{proposition}
\label{liftt}
The lifting of $\overline{\mA_{\chi}}$ to $U(\fgl_n)^{\fz_{\fgl_n}(\chi)}$ is unique for any $\chi \in \fh$.
\end{proposition}
\begin{proof}
Let us fix some regular semisimple $\chi_1 \in \fz_{\fgl_n}(\chi)$ such that $\chi + t \cdot \chi_1$ is regular for small $t \not = 0$. Consider the curve $X \in \overline{M_{0,n+1}}$  corresponding to limit $t \to 0$.
The subalgebra corresponding to this curve is
$$\mA_{\chi} \otimes_{ZU(\fz_{\fgl_n}(\chi))} \mA_{\chi_1}.$$
If there are more than one lifting of $\overline{\mA_{\chi}}$ to $U(\fgl_n)^{\fz_{\fgl_n}(\chi)}$ then we have at least two liftings of subalgebra $\mA_X$ which is in contradiction with Theorem \ref{t2}.
\end{proof}




\subsection{Limits of Bethe subalgebras}
\label{limits}

One can describe the closure of the family of Bethe subalgebras parameterized by $C \in T^{reg}$ similarly to that for shift of argument subalgebras.
Bethe subalgebras do not change under transformations of the form $C \to a C, a \in \bc^{\times}$. So the parameter space for $B(C), C \in T^{reg}$ can be regraded as $M_{0,n+2}$, i.e a point $C = \diag(z_1, \ldots, z_n)$ corresponds to $\bc\bp^1$ with an ordered set of marked points $0, z_1, \ldots, z_n, \infty$. Let $\overline{M_{0,n+2}}$ be the Deligne-Mumford compactification, see Section \ref{compact}.  

\begin{theorem} (\cite[Theorem 5.1]{ir2})
\begin{enumerate}
\item The closure of the family of Bethe subalgebras corresponding to $C \in T^{reg}$ is parameterized by $\overline{M_{0,n+2}}$.
\item Limit  subalgebras are free, maximal commutative and of the same size as $B(C)$ with $C \in T^{reg}$.
\end{enumerate}
\end{theorem}

\begin{remark}
\emph{In Section 3 we discuss Losev-Manin $\overline{L_{n+2}}$ space such that $\overline{M_{0,n+2}}$ is its blow-up. In \cite{ir3} the family of Bethe subalgebras $B(C)$ is extended to the De Concini-Procesi wonderful closure $\overline{PGL_n}$ of the adjoint group $PGL_n$. The Losev-Manin space $\overline{L_{0,n+2}}$ is the closure of the torus $T$ inside $\overline{PGL_n}$ and its boundary points correspond to the first type of limits below. This means that one can construct limit subalgebras in two steps, firstly extend to parameter space to the closure of group and then blow up subspaces where corresponding subalgebras have smaller size to obtain a flat family of subalgebras parameterized be a complete variety.}
\end{remark}

In \cite{ir2} we describe explicitly the limit subalgebra corresponding to a point $X \in \overline{M_{0,n+2}}$. We will extensively use this description so we reproduce it here.  
Let $X_\infty$ be the irreducible component of $X\in\overline{M_{0,n+2}}$ containing the marked point $\infty$. We identify $X_\infty$ with the standard $\mathbb{CP}^1$ in such a way that the marked point $\infty$ is $\infty$ and the point where the curve containing the marked point $0$ touches $X_\infty$ is $0$. To any distinguished point $\alpha\in X_\infty$ we assign the number $k_\alpha$ of nonzero marked points on the maximal (possibly reducible) curve $X_\alpha$ attached to $X_\infty$ at $\alpha$ (we set $k_\alpha=1$ if $X_\alpha$ is a (automatically, marked) point). Let $C$ be the diagonal $(n-k_0)\times(n-k_0)$-matrix with the eigenvalues $\alpha$ of multiplicity $k_\alpha$ for all distinguished points $0\ne\alpha\in X_\infty$. Then the subalgebra $i_{k_0}(B(C))$ centralized by the Lie subalgebra $\bigoplus\limits_{\alpha\ne0}\fgl_{k_\alpha}$ in $\fgl_{n-k_0}\subset i_{k_0}(Y(\fgl_{n-k_0}))\subset Y(\fgl_n)$ and by the complement sub-Yangian $\psi_{n-k_0}(Y(\fgl_{k_0}))$.
\begin{theorem}(\cite[Theorem 5.2]{ir2})

\label{result1} \begin{enumerate} \item The limit Bethe subalgebra corresponding to the curve $X\in\overline{M_{0,n+2}}$ is the product of the following $3$ commuting subalgebras: first, $i_{k_0}(B(C))\subset i_{k_0}(Y(\fgl_{n-k_0}))\subset Y(\fgl_n)$, second, the subalgebra corresponding to $X_0$ in the complement sub-Yangian $\psi_{n-k_0}(Y(\fgl_{k_0}))\subset Y(\fgl_n)$ and third, the limit shift of argument subalgebras $\mA_{X_\alpha}$ in $U(\fgl_{k_\alpha})\subset i_{k_0}(Y(\fgl_{n-k_0}))\subset Y(\fgl_n)$ for all distinguished points $\alpha\ne0$ (again we define the point $\infty$ on each $X_\alpha$ just as the intersection with $X_\infty$).
\item $i_{k_0}(B(C))$ contains the center of every $U(\fgl_{k_\alpha})\subset i_{k_0}(Y(\fgl_{n-k_0}))$. The above product is in fact the tensor product $$\psi_{n-k_0}(B(X_0))\otimes_{\bc}i_{k_0}(B(C))\otimes_{ZU(\bigoplus_{\alpha\ne0} \fgl_{k_\alpha})}\bigotimes_{\alpha\ne0}\mA_{X_\alpha}.$$
\end{enumerate}
\end{theorem}

Again as in the case of shift of argument subalgebras there are simple limits such that any limit can be obtained by its iteration.

\begin{theorem}(\cite[Theorem 5.3]{ir2})
\label{result2}
1) Let $C_0=\diag(a_1, \ldots, a_{n-k})$ and $C_1=\diag(b_1, \ldots, b_k)$.  Suppose that
$C(t) = \diag(C_0, \underbrace {0, \ldots, 0}_k)
+ t\cdot\diag(\underbrace{0, \ldots, 0}_{n-k}, C_1) \in T^{reg}$ (particulary both $C_0$ and $C_1$ are regular and non-degenerate). Then
$$\lim_{t \to 0} B(C(t)) = i_{k}(B(C_0)) \otimes \psi_{n-k}(B(C_1)).$$ \\
2) Let $C_0=\diag(\underbrace{a_1, \ldots, a_1}_{k_1}, \ldots, \underbrace{a_l, \ldots, a_l}_{k_l})$ be a non-degenerate matrix and $C_i=\diag(\underbrace{b_{i,1}, \ldots, b_{i,k_i}}_{k_i})$ for $i=1,\ldots,l$ such that $a_r\ne a_s$ and $b_{i,r}\ne b_{i,s}$ for $r\ne s$. Let
\begin{eqnarray*}C(t) := C_0 + t \cdot \diag(C_1, \ldots, C_l).\end{eqnarray*}
Then
$$\lim_{t \to 0} B(C(t)) = B(C_0) \otimes_{\bigotimes\limits_{i=1}^lZ(U(\fgl_{k_i}))} \bigotimes\limits_{i=1}^l \mA_{C_i},$$
where $\mA_{C_i}$ is the shift of argument subalgebra in $U(\fgl_{k_i}) \subset U(\fgl_n) \subset Y(\fgl_n)$ (the copy of $U(\fgl_n)$ in the Yangian $Y(\fgl_n)$ is generated by $t_{ij}^{(1)}$).
\end{theorem}

The last Theorem allows us to describe explicitly the image of $B(X), X \in \overline{M_{0,n+2}}$ under the evaluation homomorphism.



\begin{proposition}
\label{eval}
For any $X \in \overline{M_{0,n+2}}$ we have
$$\ev(B(X)) = \mA_{\sigma_1(X)},$$
where $\sigma_1$ is the involution (\ref{inv}) on $\overline{M_{0,n+2}}$ .

\end{proposition}
\begin{proof}
It is enough to prove the Proposition for the two types of simple limits.
Let $n=k_1+ \ldots +k_l$, where $k_1, \ldots, k_l \in \bz_{\geq 1}$ and
let $C_0=\diag(\underbrace{a_1, \ldots, a_1}_{k_1}, \ldots, \underbrace{a_l, \ldots, a_l}_{k_l})$ and $C_i=\diag(\underbrace{b_{i,1}, \ldots, b_{i,k_i}}_{k_i})$ for $i=1,\ldots,l$ such that $a_r\ne a_s$ and $b_{i,r}\ne b_{i,s}$ for $r\ne s$. 
For the second type of limits, the statement follows from Proposition \ref{quant}:
$$ B(C_0) \otimes_{\bigotimes\limits_{i=1}^l Z(U(\fgl_{k_i}))} \bigotimes\limits_{i=1}^l \mA_{C_i}$$
maps to $\mA_{C_0^{-1}} \otimes_{\bigotimes\limits_{i=1}^lZ(U(\fgl_{k_i}))} \bigotimes\limits_{i=1}^l \mA_{C_i}$. 

Let $C_0=\diag(a_1, \ldots, a_{n-k})$ and $C_1=\diag(a_{n-k+1}, \ldots, a_n)$.  Suppose that
$\diag(C_0, \underbrace {0, \ldots, 0}_k)
+ t\cdot\diag(\underbrace{0, \ldots, 0}_{n-k}, C_1) \in T^{reg}$.
For the first type of limits we need to consider a limit subalgebra of the form $$i_k(B(C_0)) \otimes_{\bc} \psi_{n-k}(B(C_1))$$
and prove that it maps to 
$$\mA_{C_0^{-1}} \otimes_{ZU(\fgl_k)} \mA_{(\underbrace{0, \ldots,0}_k, C_1^{-1})}.$$

By Proposition \ref{quant}, the subalgebra $i_k(B(C_0))$ maps to $\mA_{C_0^{-1}} \subset U(\fgl_{n-k})$.
We need to prove that $\psi_{n-k}(B(C_1))$ maps to $\mA_{(0, \ldots, 0, C_1^{-1})}$. According to the proof of Theorem 6.5 in \cite{ir3}, the coefficients of the following series are the generators of $\psi_{n-k}(B(C_1))$:
$$t_{1 \ldots n-k}^{1 \ldots n-k}(u+n-k)^{-1} \cdot   \left(\prod_{i=1}^{n-k} a_i\right) \sum_{1 \leqslant a_1 \leqslant k}  a_{n-k+a_1} t_{1, \ldots, n-k, n-k+a_1}^{1, \ldots, n-k, n-k+a_1}(u)$$
$$t_{1 \ldots n-k}^{1 \ldots n-k}(u+n-k)^{-1} \cdot  \left(\prod_{i=1}^{n-k} a_i\right) \sum_{1 \leqslant a_1 < a_2 \leqslant k}  a_{n-k+a_1} a_{n-k+a_2} t_{1, \ldots
, n-k, n-k+a_1, n-k+a_2}^{1, \ldots, n-k, n-k+a_1, n-k+a_2}(u)$$
\begin{center}\ldots\end{center}
$$\tau_{n}(u,X) = t_{1 \ldots n-k}^{1 \ldots n-k}(u+n-k)^{-1} \cdot a_{1} \ldots a_n t_{1 \ldots n}^{1 \ldots n}(u).$$

The coefficients of $t_{1 \ldots n-k}^{1 \ldots n-k}(u+n-k)^{-1}$ map to some elements in the center of $U(\fgl_k)$ hence to the elements of $\mA_{C_0^{-1}}$.

Therefore the images of coefficients of the following series are contained in the image of the corresponding Bethe subalgebra:

$$a_{n-k+1}^{-1} \ldots a_n^{-1} \sum_{1 \leqslant a_1 \leqslant k}  a_{n-k+a_1} t_{1, \ldots, n-k, n-k+a_1}^{1, \ldots, n-k, n-k+a_1}(u)$$
$$a_{n-k+1}^{-1} \ldots a_n^{-1} \sum_{1 \leqslant a_1 < a_2 \leqslant k}  a_{n-k+a_1} a_{n-k+a_2} t_{1, \ldots
, n-k, n-k+a_1, n-k+a_2}^{1, \ldots, n-k, n-k+a_1, n-k+a_2}(u)$$
\begin{center}\ldots\end{center}
$$t_{1 \ldots n}^{1 \ldots n}(u) = \qdet T(u).$$

The leading term of these elements are the derivatives of coefficients of characteristic polynomial along $(0, \ldots, 0, C_0^{-1})\in \fh$. Therefore the images of coefficients of these series generate $\mA_{(0, \ldots, 0, C_1^{-1})}$.

\end{proof}

\subsection{The family of shift of argument subalgebras in centralizer}
Let $k \geq 0$. 
Consider the family $\mathfrak{U}$ of shift of argument subalgebras in $S(\fgl_{n+k})^{\fgl_k}$ of the form $ \mA_{(\chi,\underbrace{0, \ldots, 0}_k)}, \chi \in \fh^{reg}$. 
\begin{proposition}(\cite[Lemma 4.11]{ir2})
\begin{enumerate}
    \item The closure of the parameter space of the family $\mathfrak{U}$ is isomorphic to $\overline{M_{0,n+2}}$.
    \item For any $X \in \overline{M_{0,n+2}}$ subalgebra $\mA_X$ is free and maximal commutative subalgebra of $S(\fgl_{n+k})^{\fgl_k}$.
\end{enumerate}
\end{proposition}
Similarly to the proof of Proposition \ref{liftt} one can prove that for any $X \in \overline{M_{0,n+2}}$ a lifting of $\mA_X \subset S(\fgl_{n+k})^{\fgl_k}$ to $U(\fgl_{n+k})^{\fgl_k}$ is unique. Moreover, the closure of the family of subalgebras of the form $\mA_{(\chi, \underbrace{0, \ldots, 0})} \subset U(\fgl_{n+k})^{\fgl_k}, \chi \in \fh^{reg}$ is isomorphic to $\overline{M_{0,n+2}}$ as well because the size of a limit subalgebra from this family is the same as for generic one, i.e. the proof is analagous to the proof of Theorem \ref{clo}.
Let us describe explicitly the  subalgebra $B(X) \subset U(\fgl_{n+k})^{\fgl_k}$ for any $X \in \overline{M_{0,n+2}}$.

Let $X_\infty$ be the irreducible component of $X\in\overline{M_{0,n+2}}$ containing the marked point $\infty$. We identify $X_\infty$ with the standard $\mathbb{CP}^1$ in such a way that the marked point $\infty$ is $\infty$ and the point where the curve containing the marked point $0$ touches $X_\infty$ is $0$. To any distinguished point $\alpha\in X_\infty$ we assign the number $k_\alpha$ of nonzero marked points on the maximal (possibly reducible) curve $X_\alpha$ attached to $X_\infty$ at $\alpha$ (we set $k_\alpha=1$ if $X_\alpha$ is a (automatically, marked) point). Let $\chi$ be the diagonal $(n-k_0)\times(n-k_0)$-matrix with the eigenvalues $\alpha$ of multiplicity $k_\alpha$ for all distinguished points $0\ne\alpha\in X_\infty$. Then the subalgebra $\mA_{(\chi,\underbrace{0, \ldots, 0}_{k_0+k})}$ is centralized by the Lie subalgebra $\bigoplus\limits_{\alpha\ne0}\fgl_{k_\alpha}$ in $\fgl_{n-k_0}\subset U(\fgl_{n-k_0})\subset U(\fgl_n)$ and by the complement $U(\fgl_{k_0})$.

\begin{proposition}
\label{ustr}
The subalgebra in $U(\fgl_{n+k})^{\fgl_k}$ corresponding to $X \in \overline{M_{0,n+2}}$ has the form
$$ \mA_{X_0} \otimes_{ZU(\fgl_k)} (\mA_{(\chi,\underbrace{0, \ldots, 0}_{k_0+k})} \otimes_{ ZU(\bigoplus_{\alpha \not = 0}\fgl_{  k_\alpha})}  \bigotimes_{i=1}^l \mA_{X_{\alpha_i}})$$

\end{proposition}
\begin{proof}
Follows from the explicit description of limit subalgebras in $U(\fgl_{n+k})$.
\end{proof}

This means that one can get any limit subalgebra by iterating the following two simple limits.
The first simple limit has the form 
$$\mA_{(C_0, \underbrace{0, \ldots, 0}_{m+k})} \otimes_{ZU(\fgl_k)} \mA_{(C_1, \underbrace{0, \ldots, 0}_k)} $$
The second simple limit has the form 
$$\mA_{(C_0, \underbrace{0,\ldots,0}_k)} \otimes_{\bigotimes\limits_{i=1}^l Z(U(\fgl_{k_i}))} \bigotimes_{i=1}^{l} \mA_{C_i}$$

\subsection{Centralizer construction and Bethe subalgebras}

Let $X \in \overline{M_{0,n+2}}$.
We have the following 
\begin{proposition} \cite{ir2}
\label{bethecenter}
The maps $\eta_k$ are an asymptotic isomorphism between $B(X) \otimes A_0$ and $\mA_{X} \subset U(\fgl_{n+k})^{\fgl_k}$.
\end{proposition}

In fact the statement of Proposition \ref{bethecenter} can be refined. 
\begin{proposition}
\label{isomor}
For any $k \geq 0$ we have $\eta_k(B(X) \otimes A_0) = \mA_{X} \subset U(\fgl_{n+k})^{\fgl_k}.$
\end{proposition}
\begin{proof}
Let $C \in T$. 
\begin{lemma}\label{le:eta}
\begin{enumerate}
    \item $\eta_k(\mA_C \otimes 1) = \mA_C \subset U(\fgl_n) \subset U(\fgl_{n+k})$
    \item $\eta_k(B(C) \otimes A_0) = \mA_{(C, \underbrace{0, \ldots, 0)}_k}$
\end{enumerate}
\end{lemma}
\begin{proof}
1) $\eta_k(\mA_C \otimes 1) = \ev \circ \omega_{n+k} \circ i_k(\mA_C) = \mA_C \subset U(\fgl_n) \subset U(\fgl_{n+k})$, because the restriction of $\omega_{n+k}$ to $U(\fgl_{n+k})$ is the identity map. \\
2) $\eta_k(B(C) \otimes A_0) = \ev \circ \omega_{n+k} \circ i_k(B(C)) \cdot ZU(\fgl_{n+k}) = \ev \circ \omega_{n+k} \circ i_k \circ \omega_n (B(C^{-1})) \cdot ZU(\fgl_{n+k}) = \ev \circ \psi_{k}(B(C^{-1})) \cdot ZU(\fgl_{n+k}) = \mA_{(C, 0, \ldots, 0)}$ according to the proof of Proposition \ref{eval}.
\end{proof}
As in the proof of Proposition \ref{eval} it is enough to show that the map is compatible with the elementary limits.
Consider a limit of the second type. Suppose that $X \in \overline{M_{0,n+2}}$ such that
$$B(X) =  B(C_0) \otimes_{\bigotimes\limits_{i=1}^lZ(U(\fgl_{k_i}))} \bigotimes \limits_{i=1}^l \mA_{C_i}$$
From Lemma it follows that $$\eta_k(B(X) \otimes A_0) = \mA_{(C_0, \underbrace{0, \ldots, 0)}_k} \otimes_{\bigotimes\limits_{i=1}^lZ(U(\fgl_{k_i}))} \mA_{C_i} = \mA_{X} \subset U(\fgl_{n+k})^{\fgl_k}$$

Consider the first type of a limit subalgebra. Let $X \in \overline{M_{0,n+2}}$ such that
$$B(X) = i_m(B(C_0)) \otimes \psi_{n-m}(B(C_1)).$$ From \cite[Lemma 2.15]{ir2} it follows, that the restriction of $\eta_k$ to $i_m(Y(\fgl_{n-m}))$ is $\eta_{k+m}$ and that the restriction of $\eta_k$ to $\psi_{n-m}(Y(\fgl_m))$ is $\eta_k$. Applying Lemma~\ref{le:eta} we get

$$\eta_k(B(X) \otimes A_0) = \mA_{(C_0, \underbrace{0, \ldots, 0)}_{k+m}} \otimes \mA_{( C_1, \underbrace{0, \ldots, 0}_k)} = \mA_{X} \subset U(\fgl_{n+k})^{\fgl_k}$$

\end{proof}

\subsection{Spectra of quantum shift of argument subalgebras}

Let $V_{\lambda}$ be an irreducible representation of $\fgl_n$ with the arbitrary highest weight $\lambda$.
\begin{theorem}
For any $X \in \overline{M_{0,n+1}}$ the subalgebra $\mA_{X}$ acts with a cyclic vector in $V_\lambda$. 
\end{theorem}

Let $k \geq 0$ and let $V_{\lambda}$ be an irreducible representation of $\fgl_{n+k}$ with the highest weight $(\lambda_1, \ldots, \lambda_{n+k})$.

Let 
$$V = \bigoplus_{\mu} M_{\lambda \mu} \otimes V_{\mu}$$
be the decomposition of the restriction of $V$ to $\fgl_k$ into irreducible representations of $\fgl_k$.
The following Proposition is well-known.

\begin{proposition}
\label{repr}
1) $M_{\lambda \mu} \not = 0$ if and only if $\mu$ is a subdiagram of $\lambda$ and
$\mu_i - \lambda_{i+n} \geq 0 \ \forall \, i = 1, \ldots, n$; \\
2) $M_{\lambda \mu}$ is an irreducible representation of $U(\fgl_{n+k})^{\fgl_k}$.
\end{proposition}

The following proposition is a refined version of Theorem 11.2, \cite{hkrw}.
\begin{proposition}
\label{c1}
For any $X \in \overline{M_{0,n+2}}$ subalgebra $\mA_{X} \subset U(\fgl_{n+k})^{\fgl_k}$ has cyclic vector on any irreducible representation of $U(\fgl_{n+k})^{\fgl_k}$ of the form $M_{\lambda \mu}$ for some integer $\fgl_{n+k}$-weight $\lambda$ and integer $\fgl_k$-weight $\mu$.
\end{proposition}
\begin{proof}
For $X \in T^{reg}$ the statement is a consequence of Theorem 11.2 \cite{hkrw}.
For other $X$ the statement of the  Proposition follows from Proposition \ref{ustr} using the induction on the number of irreducible components of the curve.
\end{proof}

\begin{proposition}
\label{c2}
For any $X \in \overline{M_{0,n+2}^{split}}$ subalgebra $\mA_{X} \subset U(\fgl_{n+k})^{\fgl_k}$ acts with simple spectrum on any irreducible representation of $U(\fgl_{n+k})^{\fgl_k}$ of the form $M_{\lambda \mu}$ for some integer $\fgl_{n+k}$-weight $\lambda$ and some integer $\fgl_k$-weight $\mu$.
\end{proposition}
\begin{proof}
For $X \in T^{reg}_{split}$ we know that there is a cyclic vector on such representation. Moreover $\mA_X = \mA_{X_0} \cap U(\fgl_{n+k})^{\fgl_k} \subset U(\fgl_{n+k})$ for some $X_0 \in \overline{M_{0,n+k+2}}$. We know \cite{ffr} that $\mA_{X_0}$ acts by Hermitian operators with respect to \emph{some} positive definite Hermitian form. This implies that $\mA_X$
acts by Hermitian operators as well with respect to restriction of this form to the multiplicity space. So  the statement follows.

For other $X \in \overline{M_{0,n+2}^{split}}$, the statement of the Proposition follows from Proposition \ref{ustr} using the induction on the number of irreducible components of the curve.
\end{proof}

\section{Cyclic vector}
\label{cyclic}
\subsection{General conjecture.} Let $X \in \overline{M_{0,n+2}}$ and consider $B(X) \subset Y(\fgl_n)$. 
Let $V_{\lambda_i \setminus \mu_i}(z_i), i =1, \ldots, k$ be a set of skew representations of $Y(\fgl_n)$. Consider the representation $W = \bigotimes_{i=1}^k V_{\lambda_i \setminus \mu_i}(z_i)$ such that $z_i - z_j \not \in \bz$.


\begin{conj}
\label{conj}
$B(X)$ has a cyclic vector in $W$  for all $X \in \overline{M_{0,n+2}}$.
\end{conj}

The following observations support the conjecture. Note that in the conjecture we do not suppose that representation $W$ is irreducible.

\begin{enumerate}
    \item Gelfand-Tsetlin subalgebra $H$ is $B(X)$ for $X \in \overline{M_{0,n+2}}$ being the so-called caterpillar curve.
\begin{center}
\includegraphics[scale=0.75]{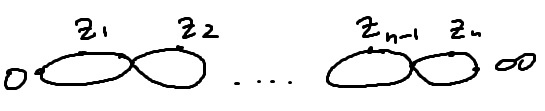}
\end{center}    
It is known that if $z_i - z_j \not \in \bz$ then $H$ acts with simple spectrum on $W$.

    \item From the geometric point of view there is an action of $Y(\fgl_n)$ on the equivariant cohomology of quivier varieties of type A. The direct sum of cohomology spaces is the tensor product of fundamental Yangian representations with the equivariant parmeters being the evaluation parameters. There is a conjecture of \cite{mo}  that $B(C), C \in T^{reg}$ acts by the operators of quantum multiplication by cohomology classes thus has the unity class as a cyclic vector;
    
    \item The $Y(\fgl_2)$-case of this conjecture is proved in \cite{mash};
    
    \item In \cite{mtv} the authors consider representation of the form $\bigotimes_{i=1}^k \bc^n(z_i)$ and prove that $B(C), C \in T^{reg}$ has a cyclic vector on it if $z_i \not = z_j + 1$ for $i>j$. Note that it includes some cases where representation is not irreducible.

\end{enumerate}

In the present paper we prove the conjecture for {\it generic} values of the parameters $z_i$. 

\subsection{Tensor product of skew representations}

We are going to prove that Bethe subalgebra of $Y(\fgl_n)$ acts with a cyclic vector on any tensor product skew representation with generic values of the evaluation parameters. From Schur lemma and Proposition \ref{tensor2} it follows that the restriction of a skew representation to $Y(\fsl_n)$ is an irreducible representation of $Y(\fsl_n)$ and that existence of a cyclic vector for a Bethe subalgebra of $Y(\fsl_n)$ implies the existence of a cyclic vector for the corresponding Bethe subalgebra of $Y(\fgl_n)$ in the original representation.

Irreducible representations of $Y(\fsl_n)$ are classified by means of {\em Drinfeld polynomials}, see e.g. \cite[Paragraph 3]{M}. 
The following Proposition describes Drinfeld polynomials of a skew representation. 
Let $\lambda \setminus \mu$ be a skew diagram such that $M_{\lambda \mu} \not = 0$. For any cell $\alpha \in \lambda \setminus \mu$ let $c(\alpha)$ be its content, i.e. $c(\alpha) = j - i$ if $\alpha$ belongs to $i$-th row and $j$-th column.
\begin{proposition}(\cite[Corollary 8.5.5]{M})
\label{poly}
Drinfeld polynomials $P_1(u), \ldots, P_{n-1}(u)$ of $V_{\lambda \setminus \mu}$ as $Y(\fsl_n)$ representation are $$P_k(u) = \prod_{\alpha}(u + c(\alpha)),$$
where the product is over all $\alpha$ be the upper cells in columns of height $k$ in $\lambda / \mu$.
\end{proposition}

\begin{proposition}
\label{skew}
A representation of the form $\bigotimes_{i=1}^k V_{\lambda_i \setminus \mu_i}(z_i)$ with $|z_i - z_j| \gg 0$ and non-positive integers $z_i$'s is a skew representation of $Y(\fsl_n)$.
\end{proposition}
\begin{proof}
It is enough to choose skew diagram $\lambda \setminus \mu $ such that $M_{\lambda \mu} \not = 0$ and Drinfeld polynomials of $V_{\lambda \setminus \mu}$ coincides with Drinfeld polynomials of  $\bigotimes_{i=1}^k V_{\lambda_i \setminus \mu_i}(z_i)$.
The following facts are well-known, see \cite{M}. \\ 1) If the tensor product of irreducible representations is irreducible then the Drinfeld polynomial of the tensor product is the product of the Drinfeld polynomials of tensor factors; \\ 
2) $V_{\lambda_i \setminus \mu_i}(z_i)$ is isomorphic to $V_{\tilde \lambda_i \setminus \tilde \mu_i}(0)$ where  $\tilde \lambda \setminus \tilde \mu$, where $\tilde \lambda_i = \lambda_i - z_i, \tilde \mu_i = \mu_i - z_i$. \newline
3) If $|z_i - z_j| \gg 0$ for all pairs $i,j$ then the tensor product $\bigotimes_{i=1}^k V_{\lambda_i \setminus \mu_i}(z_i)$ is an irreducible representation of $Y(\fsl_n)$. 

This means that it is enough to choose such Young diagrams $\lambda$ and $\mu$ that $\lambda \setminus \mu$ is a disjoint union of the skew diagrams $\tilde \lambda_i \setminus \tilde \mu_i$ such that the rows' and columns' numbers of the components are sufficiently distant from each other. Note that in this case $M_{\lambda \mu} = \bigotimes_{i=1}^k V_{\tilde \lambda_i \setminus \tilde \mu_i}(0)=\bigotimes_{i=1}^k V_{\lambda_i \setminus \mu_i}(z_i)$ as a $Y(\fsl_n)$-module.
\end{proof}

\subsection{Cyclic vector for Bethe subalgebras of $Y(\fgl_n)$}
From Proposition \ref{isomor} it follows that the image of $B(X)$ under $\eta_k$ is  $\mA_{X} \subset U(\fgl_{n+k})^{\fgl_k}$. 
%


\begin{theorem}
\label{cons}
For all $X \in \overline{M_{0,n+2}}$, the subalgebra $B(X)$ has a cyclic vector on any skew representation $V$ of $Y(\fgl_n)$.
\end{theorem}
\begin{proof}
Follows from Proposition \ref{c1}.
\end{proof}


\begin{theorem}
\label{main1}
Let $I$ be the subset of  $\bc^n$ such that all $B(X), X \in \overline{M_{0,n+2}}$ has a cyclic vector on $V$. The the set $I$ is a Zariski open subset of $\bc^n$. 
\end{theorem}
\begin{proof}
Consider the space $\overline{M_{0,n+2}} \times \bc^n$. 
Let $\pi_{\bc^n}: \overline{M_{0,n+2}} \times \bc^n \to \bc^n$ be the projection to second factor. 
Let $J \subset \overline{M_{0,n+2}} \times \bc^n$ be the set such that for any $(X,(z_1, \ldots, z_k))$ subalgebra $B(X)$ does not have a cyclic vector on $W = \bigotimes_{i=1}^k V_{\lambda_i \setminus \mu_i}(z_i)$. 

\begin{lemma}
\label{ta}
The set $\pi_{\bc^n}(J)$ is a proper closed subset (in Zariski topology) of $\bc^n$.
\end{lemma}
\begin{proof}
 Note that $\overline{M_{0,n+2}}$ is a projective algebraic variety hence is a complete algebraic variety. It follows that for any algebraic variety 
$Y$ the projection $\pi_Y: \overline{M_{0,n+2}} \times Y \to Y$ is a closed map (with respect to Zariski topology). Consider $Y = \bc^k$ and the projection $\pi_{\bc^n}: \overline{M_{0,n+2}} \times \bc^n \to \bc^n$. Note that the condition to have a cyclic vector is open hence the set $J \subset  \overline{M_{0,n+2}} \times \bc^n$ is a closed subset thus its image $\pi_{\bc^n}(J)$ is also closed. From Proposition \ref{skew} it follows that there exist parameters $z_i, i =1, \ldots, k$ such that $W$ is a skew representation of $Y(\fgl_n)$. Then from Theorem \ref{cons} it follows that $\pi_{\bc^n}(J)$ is a proper closed subset of $\bc^n$.
\end{proof}

From the Lemma it follows that $I$ is Zariski open and non-empty.

\end{proof}

\begin{corol}
For all $X \in \overline{M_{0,n+2}}$ the subalgebra $B(X)$ has a cyclic vector on any representation of the form $\bigotimes_{i=1}^k V_{\lambda_i \setminus \mu_i}(z_i)$ for all $|z_i - z_j| \gg 0$ on any irreducible curve in $\bc^n$ which is not contained in $J$.
\end{corol}
\begin{proof}
If the curve in $\bc^n$ does not belong to $J$ then $J \cap \bc^n$ is a proper closed subset of the curve thus finite.

\end{proof}

\section{Real forms of $\overline{M_{0,n+2}}$ and simplicity of spectra}
\label{6}
In this section we give a proof of modified Proposition 4.2 from \cite{kr} following the argument of Section 1 of \cite{R}.
\subsection{Hermitian product on a tensor product of evaluation representations} 

Let \(W\) be a representation of $Y(\fgl_n)$ and $\rho_W$ be the corresponding morphism to $\End(W)$. Let $T^{(p)}(u)$ be a $\binom{n}{p} \times \binom{n}{p}$ matrix whose entries are the principal quantum $p \times p$ minors. One can also define $T^{(p)}(u)$ by the following two equalities:
$$T^{(p)}(u) = A_p T_1(u) \ldots T_p(u-p+1) = T_p(u-p+1) \ldots T_1(u) A_p$$
We consider $T^{(p)}(u)$ as an element of $Mat_{\binom{n}{p}} \otimes Y(\fgl_n)[[u^{-1}]]$. By definition put \[\cTp(u)=\id\otimes\rho_W(\Tp(u)).
\]
If  \(W=V(z)\), where \(V\) is an irreducible \(\gln\) module and $z \in \bc$ then
\begin{equation*}
\label{toperatorforevrep}    
\mathcal T^{(1)}(u)=1+\left(\sum_{1\leqslant i, j\leqslant n} e_{ij}\otimes \rho_{V}(E_{ij})\right)(u+z)^{-1}
\end{equation*}
where \(e_{ij} \in \) \(Mat_n\) are matrix identities and \(E_{ij}\) form a basis of \(\gln\).

Let
\[
W=V_1(z_1)\otimes\ldots\otimes V_k(z_k)
\]
be a product of evaluation modules. Then we have
\begin{equation*}
\label{T_W as a product}
\mathcal T_W^{(1)}(u)=\prod_{i=1}^k \mathcal T_{V_i(z_i)}^{(1)}(u)
\end{equation*}

Each \(V_i\) has a standard Hermitian form which makes \(V_i\) a unitary \(\gln\) representation. By taking the product of these forms, we obtain a Hermitian form on \(W\). We denote by \(^+\) the conjugate operator with respect to the Hermitian form. By $^{t}$ we denote the transpose of the operators with respect to matrix argument, i.e. the $(i,j)$-th matrix element of $T(u)^t$ is $t_{ji}(u)$ and the same for $\cTone(u)^t$.

\begin{proposition}\label{pr:normal}
Suppose that  for each \(i\), \(1\leqslant i\leqslant k\),
\begin{equation*}
\label{normcondition}
\left(\cTone_{V_i(z_i)}(u)\right)^{-1}=f_{i}(u)\left ((\cTone_{V_i(z_i)}(-u+c))^{t}\right)^+
\end{equation*}
with some \(f_{i}(u)\in 1+ \mathbb C[[u^{-1}]]\) and \(c\in\mathbb C\)
and that $C \in T^{comp}$. Then the subalgebra \(B(C)\) acts by normal operators on \(W\).
\end{proposition}

\begin{proof}
From the definition of $T^{(p)}(u)$ it follows that for any representation $V_i(z_i)$ of $Y(\fgl_n)$ we have
$$\cTp(u) = A_k \cTone_1(u) \ldots \cTone_k(u-p+1)$$

Using the fact that $T(u) \to T^t(u), T(u) \to T(-u)$ and the conjugation are anti-automorphisms, we get
\begin{eqnarray*}
(\cTp(-u+c)^{t})^+ = A_p (\cTone_p(-u+p-1+c)^t)^+ \ldots (\cTone_1(-u+c)^t)^+ = \\
 f_i(u)^{-1}\ldots f_i(u-p+1)^{-1} A_p (\cTone_p)^{-1}(-u+p-1) \ldots (\cTone_1)^{-1}(u) = g_{pi}(u) (\cTp(u))^{-1}
\end{eqnarray*}





Since
\[
\cTp_W(u)=\prod_{i=1}^{k}\cTp_{V_i(z_i)}(u)
\]
we obtain
\[
\left(\cTp_{W}(u)\right)^{-1}=\prod_{i=1}^k g_{pi}(u)\left((\cTp_W(-u+c))^{t}\right)^+
\] for some $g_{pi}(u) \in 1+\bc[[u^{-1}]]$ and $1 \leq p \leq n$.

Abusing notation we denote by $\tau_p(u,C)$ the image of $\tau_p(u,C)$ in $W$. Multiplying the last equality by $\rho_{\Lambda^p(\bc^n)}(C^{-1}) = \rho_{\Lambda^p(\bc^n)}(\overline{C})$ and taking the trace we get that 
$$S(\tau_p({u,C^{-1}})) = \left(\prod_{i=1}^k g_{pi}(u)\right) \tau_p(-u+c,C)^+$$
But $S(B(C^{-1})) = B(C)$, see e.g \cite{ir2}. From this we see that coefficients of $\tau_k(u,C)^+$ belong to $B(C)$ hence $B(C)$ acts by normal operators.



\end{proof}

\begin{corol}
Under the conditions of Proposition~\ref{pr:normal}, the subalgebra $B(X)$ acts by normal operators on $W$ for any $X\in \overline{M_{0,n+2}^{comp}}$.
\end{corol}

\begin{proof}
Indeed, the condition $[B(X),B(X)^+]=0$ determines a Zariski closed subset in $\overline{M_{0,n+2}^{comp}}$. On the other hand, it is satisfied on a Zariski open dense subset $T_{comp}^{reg}\subset \overline{M_{0,n+2}^{comp}}$. Hence it is satisfied for any $X\in \overline{M_{0,n+2}^{comp}}$.  
\end{proof}

\subsection{} Now we are going to find representations of $Y(\fgl_n)$ which satisfy the conditions of Proposition~\ref{normcondition}. Let \(\omega_r,1\leqslant r\leqslant n-1\) be the fundamental weights of \(\sln\).

\begin{proposition}
\label{semisimp}
Let $V$ be an irreducible representation of \(\gln\) and $V(z)$ be a evaluation representation of $Y(\fgl_n)$. Then the condition of Proposition \ref{normcondition}
is equivalent to the following two conditions:

\begin{enumerate}
    \item The highest weight of $V$ is equal to \(l\omega_r\) for some \(l\in\mathbb Z\), \(1\leqslant r\leqslant n-1\);
    \item the evaluation parameter \(z=\frac{l+n-r-\bar{c}+1}{2} + iv\), with \(v\in\mathbb R\).
\end{enumerate}
Here $c$ is from Proposition \ref{normcondition}.
\end{proposition}

\begin{proof}
We can rewrite the condition from Proposition \ref{normcondition} as
\begin{equation*}
\label{normcondoneside}    
\cTone_{V(z)}(u)^{t}\cTone_{V(z)}(-u+c)^+=f(u)\mathrm{Id}\otimes\mathrm{Id}.
\end{equation*} for some $f(u) \in 1 + \bc[[u^{-1}]]$.

As we know from (\ref{toperatorforevrep}),
\[
\mathcal T_{V(z)}^{(1)}(u)=1+\left(\sum_{1\leqslant i, j\leqslant n} e_{ij}\otimes \rho_{V}(E_{ij})\right)(u+z)^{-1}.
\]
 
Putting this together and multiplying by \((u+z)(-u+\Bar{z}+\Bar{c})\), on the LHS we get
\[
(u+z)(-u+\Bar{z}+\Bar{c})+(z+\Bar{z}+\Bar{c})\left (\sum_{1\leqslant i,j\leqslant n}e_{ij}\otimes\rho_{V}(E_{ij})\right )+\sum_{1\leqslant i,j,k\leqslant n}e_{ij}\otimes \rho_{V}(E_{ik})\rho_{V}(E_{kj}).
\]

Denote \(P=\sum_{1\leqslant i,j\leqslant n}e_{ij}\otimes\rho_{V}(E_{ij}) \in \End(\bc^n) \otimes \End(V)\). Therefore for the condition (\ref{normcondoneside}) to be satisfied we need 
$$P^{2}=-(z+\Bar{z}+\Bar{c})P.$$ So the statement of the Proposition reduces to the following lemma

\begin{lemma}
Suppose that $V$ is an irreducible $\fgl_n$ representation. Then we have $P^2 = a P$ for some $a \in \bc$ if and only if the following two conditions are satisfied:
\begin{enumerate}
\item The highest weight of $V$ is $l \omega_r$ for some $l \in \bz_{>0}, 1 \leq r \leq n-1$;
\item $a = l+n-r$.
\end{enumerate}
\end{lemma}
\begin{proof}

Suppose that \(V\) has the highest weight of the form \(l\omega_r\) for some \(r,l\). Then \(P\) acts proportionally to \(P^2\) on \(\bC\otimes V\). 
We know that \(P\) acts as an endomorphism of \(\gln\) representation \(\bC\otimes V\). Hence both \(P\) and \(P^2\) are endomorphisms of this representation. If the highest weight of \(V\) equals \(l\omega_r\) for some \(r,l\), then by Littlewood-Richardson rule the product \(\bC\otimes V\) is isomorphic to the sum of two non-isomorphic irreducible representations of \(\gln\). Hence \(\End(\bC\otimes V)\) is two-dimensional and thus \(P^2=\alpha P+\beta \mathrm{Id}\). If we compare the action of \(\tr P^2\), \(\tr P\) and \(\tr \mathrm{Id}\) and then compare the action of coefficient at \(e_{11}\) for \(P^2\), \(P\) and \(\mathrm{Id}\), we get that \(\beta = 0\) and thus the claim follows.

Suppose that \(P\) acts proportionally to \(P^2\) on \(\bC\otimes V\). Note that the $i$-th diagonal element of $P$ is $E_{ii}$ and that of $P^2$ is $\sum_{j=1}^n E_{ij} E_{ji}$. 
Comparing the action of these matrix elements on the highest vector of $V$ we obtain the condition on the representation and the proportionality coefficient \(l+n-r\).

\end{proof}
Comparing the coefficients of $P$ we get $$l+n-r = -(a+\bar{a}+
\bar{c}).$$ This implies that \(a=\frac{-l-n+r-\bar{c}}{2} + iv\), \(v\in\mathbb R\).

\end{proof}

The $Y(\fgl_n)$-module $V_{l\omega_r}(z)$ is called \emph{Kirillov-Reshetikhin modules}. Its highest weight with respect to $\fgl_n$ is the rectangular Young diagram of the size $l\times r$.

\subsection{Compact real form.} 

\begin{theorem}\label{th:compact}
Suppose that all $V_i$'s are Kirillov-Reshetikhin modules. Let $l_i \times r_i$ be the size of the corresponding Young diagram.
Then, for $(x_1, \ldots, x_k)  \in \br^k$ from a Zariski open subset, subalgebra $B(X)$ for all $X \in \overline{M_{0,n+2}^{comp}}$ has simple spectrum on $\bigotimes_{i=1}^k V_i\left( \frac{l_i}{2} - \frac{r_i}{2} + i x_i\right)$. 

\end{theorem}
\begin{proof}
From Lemma \ref{ta} we know that for all $X \in \overline{M_{0,n+2}}$ there is a Zariski open subset $I$ of $\bc^k$ such that $B(X)$ for all $X \in \overline{M_{0,n+2}}$  has a cyclic vector on $\bigotimes_{i=1}^k V_i(z_i)$. Then $I \cap \br^k$ is a Zariski open subset of $\br^k$ such that for all $X \in \overline{M_{0,n+2}}$ the subalgebra $B(X)$ has a cyclic vector on $\bigotimes_{i=1}^k V_i\left( \frac{l_i}{2} - \frac{r_i}{2} + i x_i\right)$. If we restrict ourselves to $X \in \overline{M_{0,n+2}^{comp}}$ then we also have the semismplicity of action of $B(X)$ from Proposition \ref{semisimp} which concludes the proof.
\end{proof}

\subsection{Split real form.}

\begin{theorem}\label{th:split}
Let $V_i, i = 1, \ldots, k$ be a set of skew representations of $Y(\fgl_n)$. Then, for $(x_1, \ldots, x_k)  \in \br^k$ from a Zariski open subset, the subalgebra $B(X)$ has simple spectrum on $\bigotimes_{i=1}^k V_i(x_i)$ for all $X \in \overline{M_{0,n+2}^{split}}$.
\end{theorem}
\begin{proof}
The proof is analogous to the proof of Theorem \ref{main1}. There are parameters $(x_1, \ldots, x_n)$ such that $\bigotimes_{i=1}^k V_i(x_i)$ is isomorphic to a skew representation. Thus for such parameters for any $X \in \overline{M_{0,n+2}^{split}}$ subalgebra $B(X)$ acts as $\mA_X \subset U(\fgl_{n+k})^{\fgl_k}$, see Proposition \ref{isomor}. But the latter subalgebra acts with simple spectrum, see Proposition \ref{semisimp}. It means that we find at least one set of parameters $(x_1, \ldots, x_k)$ such that for all $X \in \overline{M_{0,n+2}^{split}}$ subalgebra $B(X)$ acts with simple spectrum. Using the fact that $\overline{M_{0,n+2}^{split}}$ is projective variety and hence complete we obtain the result analogous to Theorem \ref{main1}.

\end{proof}

\section{Further development and conjectures}
\label{further}
The cyclicity property implies that there are only finitely many joint eigenlines for the Bethe subalgebra $B(X)$ in the given $Y(\fgl_n)$-module. Following \cite{hkrw,R14} we can regard these joint eigenlines at once as a \emph{covering} over the Deligne-Mumford space $\overline{M_{0,n+2}}$ (possibly, ramified, as the operators from $B(C)$ can sometimes have non-trivial Jordan blocks). The Hermitian property implies that this covering is unramified over the corresponding real form of $\overline{M_{0,n+2}}$, so is determined by the action of the fundamental group on some fiber. We expect that this action of the fundamental group has a combinatorial description in terms of partial Schutzenberger involutions on Kirillov-Reshetikhin crystals in the compact case an in terms of Bender-Knuth involutions on Gelfand-Tsetlin polytopes in the split case. Let us give a more precise description of the combinatorial structures arising in these two cases.

One defines the fundamental groupoid of $\overline{M_{0,n+2}^{split}}=\overline{M_{0,n+2}}(\br)$ as follows. The objects are the components of the open stratum which are given by fixing the increasing order of the $z_i$'s. The mophisms are the homotopy classes of paths which connect some inner points of the components of the open stratum and cross the strata of codimension $1$ transversely. This groupoid is in fact the orbifold fundamental group of $\overline{M_{0,n+2}}(\br)/S_{n+1}$ usually denoted by $J_{n+1}$. Clearly, the group $J_{n+1}$ is generated by the homotopy classes of paths connecting \emph{neighboring} open cells (i.e. the open cells having common face of codimension $1$). Thus there are the following generators of $J_{n+1}$.

For positive integers $p\le q$, denote by $[p,q]$ the set $\{p,p+1,\ldots,q-1,q\}$. Let $\overline{s_{p,q}}\in S_{n+1}$ be the involution reversing the segment $[p,q]\subset[0,n]$. Denote by $s_{p,q}$ the element of $J_n$ corresponding to the shortest path connecting the cells which differ by $\overline{s_{p,q}}$ (they have a common face). Then the elements $s_{p,q}$ with $0\le p<q\le n$ generate $J_{n+1}$ and the defining relations are
\begin{equation*}
\begin{array}{l}s_{p,q}^2=e;\\
s_{p_1,q_1}s_{p_2,q_2}=s_{p_2,q_2}s_{p_1,q_1}\ \text{if}\ q_1<p_2;\\
s_{p_1,q_1}s_{p_2,q_2}s_{p_1,q_1}=s_{p_1+q_1-q_2,p_1+q_1-p_2}\ \text{if}\ p_1\le p_2<q_2\le q_1.
\end{array}
\end{equation*}
We refer the reader to \cite{hk} for more details.

The fundamental group $PJ_{n+1}:=\pi_1(\overline{M_{0,n+1}(\br)})$ is the kernel of the natural homomorphism $J_{n+1}\to S_{n+1}$ which maps $s_{p,q}$ to $\overline{s_{p,q}}$. By the analogy with braid groups, $PJ_{n+1}$ is called the \emph{pure cactus group}.

Similarly, one can define the fundamental groupoid for $\overline{M_{0,n+2}^{comp}}$ whose generators correspond to flipping a segment in the cyclic ordering of the marked points $z_1,\ldots,z_n$ and the relations are similar. It is natural to call this \emph{affine} cactus group.

\subsection{Compact case.}

According to Theorem~\ref{th:compact} for any collection of Kirillov-Reshetikhin modules $V_i$ there is a Zariski open dense subset $U\subset \bc^N$ such that for any $(z_1,\ldots,z_N)\in U$ the algebra $B(X)$ acts with simple spectrum on the tensor product $V_1(z_1)\otimes\ldots\otimes V_N(z_N)$ for all $X\in \overline{M_{0,n+2}^{comp}}$. In \cite{kmr} we define the structure of a Kirillov-Reshetikhin crystal on this spectrum following the strategy of \cite{hkrw}. One can write the generators of the fundamental groupoid of $\overline{M_{0,n+2}^{comp}}$ as the paths connecting two neighboring components of the open stratum which cross transversely the codimension~$1$ stratum at a generic point. Such generators correspond to the longest elements of Weyl groups for all connected Dynkin subdiagrams in the affine Dynkin diagram.     

We expect that the following is true:

\begin{conj}
The above generators of the fundamental groupoid of $\overline{M_{0,n+2}^{comp}}$ act on this KR crystal via partial Schutzenberger involutions assigned to any connected subdiagram of the affine Dynkin diagram $\tilde{A}_{n+1}$. 
\end{conj} 


\subsection{Split case.} Let $V$ be a tensor product of skew representations of the Yangian $Y(\fgl_n)$ with sufficiently distant real parameters, as in Theorem~\ref{th:split}. Then according to Theorem~\ref{th:split} the algebra $B(X)$ acts on $V$ with simple spectrum for all $X\in \overline{M_{0,n+2}(\br)}$. This gives an unramified covering $\mathcal{E}_V$ over $\overline{M_{0,n+2}(\br)}$ whose fiber over $X\in \overline{M_{0,n+2}(\br)}$ is the set of eigenlines for $B(X)$ on $V$. 

\begin{remark}
\emph{This agrees with the results of Nazarov and Tarasov \cite{nt} stating that the Gelfand-Tsetlin subalgebra (or the Cartan subalgebra of the Yangian) generated by the elements $A_i(u)$ in Drinfeld's new realization act with simple spectrum on the same class of representations. Namely, the Gelfand-Tsetlin subalgebra $H$ is $B(X_0)$ for $X_0$ being the caterpillar point of $\overline{M_{0,n+2}(\br)}$.}
\end{remark}

The eigenbasis for $H=B(X_0)$ in a skew representation of $Y(\fgl_n)$ is indexed by keystone Gelfand-Tsetlin patterns (or, equivalently, by semistandard skew Young tableaux). So it is convenient to choose $X_0$ as the base point and regard the monodromy of the eigenlines as combinatorial transformations of the Gelfand-Tsetlin patterns. Moreover, we can assign to any connected component of the open stratum of $\overline{M_{0,n+2}^{split}}$ a caterpillar point in its closure, the only one with the given ordering of marked points. We expect that the following is true: 

\begin{conj}
The paths along $1$-dimensional strata connecting the caterpillar points generate the fundamental group $J_{n+1}$ and the action of these generators on the set of semistandard skew tableau is given by Bender-Knuth involutions (see \cite{bk} for definitions).
\end{conj} 

In \cite{cgp} Chmutov, Glick and Pylyavskyy show that Bender-Knuth involutions generate some quotient of the cactus group. We hope that the above conjecture gives the geometric explanation of this fact, which is just rewriting the cactus groupoid in dual terms (of vertices and edges instead of the open cells and codimension one faces).


\newpage
\bigskip
\noindent \footnotesize{{\bf Aleksei Ilin} \\
National Research University 
Higher School of Economics, \\
Russian Federation,\\
Department of Mathematics, 6 Usacheva st, Moscow 119048;\\
{\tt alex.omsk2@gmail.com}} \\
\\
\footnotesize{
{\bf Inna Mashanova-Golikova} \\
National Research University
Higher School of Economics,\\ Russian Federation,\\
Department of Mathematics, 6 Usacheva st, Moscow 119048 \\
{\tt inna.mashanova@gmail.com}} \\
\\
\footnotesize{
{\bf Leonid Rybnikov} \\
National Research University
Higher School of Economics,\\ Russian Federation,\\
Department of Mathematics, 6 Usacheva st, Moscow 119048;\\
Institute for Information Transmission Problems of RAS;\\
{\tt leo.rybnikov@gmail.com}}

\end{document}